\documentclass[11pt,reqno]{amsart}
\usepackage[dvipdfm,a4paper,top=30mm,bottom=25mm,left=25mm,right=25mm]{geometry}
\usepackage{amsmath,amssymb,mathtools,cite}
\usepackage{comment,mathrsfs}

\usepackage[colorlinks=true,linkcolor=blue,citecolor=blue]{hyperref}

\usepackage{graphicx}

\newtheorem{theorem}{Theorem}[section]
\newtheorem{lemma}[theorem]{Lemma}
\newtheorem{proposition}[theorem]{Proposition}
\newtheorem{corollary}[theorem]{Corollary}

\theoremstyle{definition}
\newtheorem{definition}[theorem]{Definition}

\theoremstyle{remark}
\newtheorem{remark}[theorem]{Remark}

\numberwithin{equation}{section}

\allowdisplaybreaks[4]

\def\rnum#1{\expandafter{\romannumeral #1}} 
\def\Rnum#1{\uppercase\expandafter{\romannumeral #1}}

\renewcommand{\Re}{{\operatorname{Re}\,}} 
\renewcommand{\Im}{{\operatorname{Im}\,}} 
\def\rad{\text{rad}}

\newcommand{\e}{{\varepsilon}}

\newcommand{\N}{{\mathbb{N}}} 
\newcommand{\R}{{\mathbb{R}}} 

\def\~#1{\widetilde #1}

\def\({\left(}
\def\){\right)}
\def\[{\left[}
\def\]{\right]}
\def\<{\left\langle}
\def\>{\right\rangle}

\begin{document}

\title[]{Scattering and blow-up solutions to nonlinear Schr\"odinger equation beyond the threshold}


\author{Masaru Hamano}
\address{Center for Science Adventure and Collaborative Research Advancement, Graduate School of Science, Kyoto University, Kyoto 606-8502, Japan.}
\email{hamano.masaru.3c@kyoto-u.ac.jp}









\begin{abstract}
In this paper, we consider nonlinear Schr\"odinger equation.
Duyckaerts--Roudenko [Comm. Math. Phys. \textbf{334} (2015), no. 3, 1573--1615] investigated time behavior of solutions to the equation, whose mass-energy is greater than that of the ground state.
We revisit the results and show by using the action.
Moreover, we apply it to other equations without scale invariance.
\end{abstract}

\maketitle

\tableofcontents


\section{Introduction}

In this paper, we consider nonlinear Schr\"odinger equation
\begin{align}\label{NLS}
	\left\{ \hspace{-0.2cm}
	\begin{array}{l}
		\displaystyle
		i\partial_t u + \Delta u
			= - |u|^{p-1}u, \quad (t,x) \in \R \times \R^d, \\
		\displaystyle
		u(0,\cdot)
			= u_0
			\in H^1(\R^d),
	\end{array}
	\right.
\end{align}
where
$
	1+\frac{4}{d}
		< p
		< 2^\ast-1
$
for
$
	2^\ast
		:= \infty
$
$
	(d = 1,2)
$
and
$
	2^\ast
		:= \frac{2d}{d-2}
$
$
	(d \geq 3).
$
This equation has the following scaling property:
If
$
	u
$
is a solution to \eqref{NLS}, then
$
	u_\lambda
		:= \lambda^\frac{2}{p-1}u(\lambda^2\cdot, \lambda\cdot)
$
$
	(\lambda > 0)
$
is also a solution to \eqref{NLS}.
Under this scaling, we have
$
	\|u_\lambda(0)\|_{\dot{H}^{s_c}}
		= \|u(0)\|_{\dot{H}^{s_c}}
$
for
\begin{align}\label{Sc}
	s_c
		:= \frac{d}{2} - \frac{2}{p-1}.
\end{align}
Therefore,
$
	p
		= 1 + \frac{4}{d}
$
($
	\Longleftrightarrow
	s_c
		= 0
$)
is called
$
	L^2
$-critical (or mass-critical) exponent and
$
	p
		= 2^\ast-1
$
($
	\Longleftrightarrow
	s_c
		= 1
$)
is called
$
	H^1
$-critical (or energy-critical) exponent.
$
	H^1
$-solution exists, at least locally in time under
$
	1
		< p
		< 2^\ast-1
$
and it conserves its mass and energy defined as
\begin{align*}
	M(f)
		:= \|f\|_{L^2}^2
	\ \text{ and }\ 
	E(f)
		:= \frac{1}{2}L(f) - \frac{1}{p+1}N(f),
\end{align*}
where
\begin{align}\label{176}
	L(f)
		:= \|\nabla f\|_{L^2}^2
	\ \text{ and }\ 
	N(f)
		:= \|f\|_{L^{p+1}}^{p+1}.
\end{align}
Moreover, the following blow-up criterion holds:
For the maximal existence
$
	T_{\max},
$
we have either
$
	T_{\max}
		= \infty
$
or
``$
	T_{\max}
		< \infty
$
and
$
	\lim_{t \rightarrow T_{\max}}\|\nabla u(t)\|_{L^2}
		= \infty".
$
A corresponding property also holds for the minimal existence time
$
	T_{\min}.
$

We investigate global behavior of the solutions to \eqref{NLS}.
By the choice of initial data, the solutions to \eqref{NLS} show various behaviors, scattering, blow-up, and so on.

\begin{definition}
Let
$
	u
$
be a solution to \eqref{NLS} on the maximal existence time interval
$
	(T_{\min},T_{\max}).
$
\begin{itemize}
\item
(Scattering)
If the solution
$
	u
$
exists globally in the forward time and there exists
$
	\psi_+
		\in H^1(\R^d)
$
such that
\begin{align*}
	\lim_{t \rightarrow \infty}\|u(t) - e^{it\Delta}\psi_+\|_{H^1}
		= 0,
\end{align*}
then we call the solution
$
	u
$
scatters in the forward time.
For scattering in the backward time, the definition is similar.
We say that the solution
$
	u
$
to \eqref{NLS} scatters simply when it scatters in the forward and backward times.
\item
(Blow-up)
If the maximal forward existence time
$
	T_{\max}
$
is finite, then the solution
$
	u
$
to \eqref{NLS} blows up in the forward time.
For blow-up in the backward time, the definition is similar.
We say that the solution
$
	u
$
to \eqref{NLS} blows up simply when it blows up in the forward and backward times.
\end{itemize}
\end{definition}

Beside scattering solutions and blow-up solutions, \eqref{NLS} has also standing wave solutions
$
	u(t,x)
		= e^{i\omega t}Q_\omega(x).
$
We see that
$
	Q_\omega
$
satisfies the following elliptic equation:
\begin{align}\label{SP}
	- \omega Q_\omega + \Delta Q_\omega
		= - |Q_\omega|^{p-1}Q_\omega.
\end{align}
Corresponding functional is defined as
\begin{align}\label{Sw}
	S_\omega(f)
		:= \frac{\omega}{2}M(f) + E(f).
\end{align}
Especially, there exists the ground state
$
	Q_\omega
$
to \eqref{SP} and is characterized as the optimizer to a minimizing problem
\begin{align}\label{nw}
	n_\omega
		:= \inf\{S_\omega(f) : f \in H^1(\R^d) \setminus \{0\}, K(f) = 0\},
\end{align}
that is,
\begin{align}\label{MP}
	S_\omega(Q_\omega)
		= n_\omega
	\ \text{ and }\ 
	K(Q_\omega)
		= 0,
\end{align}
where the virial functional
$
	K
$
is defined as
\begin{align}\label{K}
	K(f)
		:= \left.\frac{d}{d\lambda}\right|_{\lambda = 0}S_\omega(e^{d\lambda}f(e^{2\lambda}\cdot))
		= 2L(f) - \frac{d(p-1)}{p+1}N(f).
\end{align}
From now on,
$
	Q_\omega
$
denotes the ground state to \eqref{SP}.
Since
$
	Q_\omega
$
is a solution \eqref{SP} and has the property \eqref{MP}, the identities
\begin{align}\label{PI}
	\omega M(Q_\omega) + L(Q_\omega)
		= N(Q_\omega)
	\ \text{ and }\ 
	2L(Q_\omega)
		= \frac{d(p-1)}{p+1}N(Q_\omega)
\end{align}
hold.
Moreover,
$
	Q_\omega
$
attains also the Gagliardo--Nirenberg inequality
\begin{align}\label{G-N inequality}
	N(f)
		\leq C_{\rm GN}M(f)^\frac{(p-1)(1-s_c)}{2}L(f)^{\frac{(p-1)s_c}{2}+1}
	\quad \text{for any }f \in H^1(\R^d),
\end{align}
where
$
	C_{\rm GN}
$
is the best constant.
This inequality deduces a framework of the following theorem given in \cite{DuWuZha15, DuyHolRou08, FanXieCaz11, HolRou08, HolRou10} by using tools enveloped in \cite{Gla77, KenMer06, OgaTsu91}:
\begin{theorem}\label{Below}
Let
$
	d
		\in \N,
$
$
	1+\frac{4}{d}
		< p
		< 2^\ast-1,
$
and
$
	u_0
		\in H^1(\R^d).
$
Assume that
$
	u_0
$
satisfy
\begin{align} \label{100}
	M(u_0)^\frac{1-s_c}{s_c}E(u_0)
		< M(Q_1)^\frac{1-s_c}{s_c}E(Q_1),
\end{align}
where
$
	s_c
$
is given in \eqref{Sc}.
\begin{itemize}
\item
(Scattering)
If
$
	M(u_0)^{1-s_c}L(u_0)^{s_c}
		< M(Q_1)^{1-s_c}L(Q_1)^{s_c},
$
then the solution
$
	u
$
to \eqref{NLS} satisfies
\begin{align*}
	\sup_{t \in \R}M(u(t))^{1-s_c}L(u(t))^{s_c}
		< M(Q_1)^{1-s_c}L(Q_1)^{s_c}
\end{align*}
and scatters.
\item
(Blow-up or grow-up)
If
$
	M(u_0)^{1-s_c}L(u_0)^{s_c}
		> M(Q_1)^{1-s_c}L(Q_1)^{s_c},
$
then the solution
$
	u
$
to \eqref{NLS} blows up or grows up, where we call
$
	u
$
grows up in the forward time when
$
	u
$
satisfies
$
	T_{\max}
		= \infty
$
and
\begin{align*}
	\limsup_{t \rightarrow \infty}\|\nabla u(t)\|_{L^2}
		= \infty.
\end{align*}
In addition, when either
$
	xu_0
		\in L^2(\R^d)
$
or
``
$
	d
		\geq 2,
$
$
	p
		\leq 5,
$
and
$
	u_0
$
is radially symmetric'' is satisfied, the solution
$
	u
$
to \eqref{NLS} blows up.
\end{itemize}
\end{theorem}

\begin{remark}\label{AkaNawCon}
The conditions in Theorem \ref{Below} can be rewritten as follows:
\begin{itemize}
\item
\eqref{100} is equivalent to there exists
$
	\omega
		> 0
$
such that
$
	S_\omega(u_0)
		< S_\omega(Q_\omega).
$
\end{itemize}
Under this condition, the next relations also hold:
\begin{itemize}
\item
$
	M(u_0)^{1-s_c}L(u_0)^{s_c}
		< M(Q_1)^{1-s_c}L(Q_1)^{s_c}
$
is equivalent to
$
	K(u_0)
		\geq 0,
$
\item
$
	M(u_0)^{1-s_c}L(u_0)^{s_c}
		> M(Q_1)^{1-s_c}L(Q_1)^{s_c}
$
is equivalent to
$
	K(u_0)
		< 0.
$
\end{itemize}
Akahori--Nawa \cite{AkaNaw13} proved Theorem \ref{Below} by using the rewritten notations.
\end{remark}

Next, our interest turns to threshold or above the ground state solutions satisfying
\begin{align*}
	M(u_0)^\frac{1-s_c}{s_c}E(u_0)
		\geq M(Q_1)^\frac{1-s_c}{s_c}E(Q_1)
\end{align*}
(which is equivalent to
$
	S_\omega(u_0)
		\geq S_\omega(Q_\omega)
$
for some
$
	\omega
		> 0
$).
For example, the threshold case in \cite{CamFarRou22, DuyRou10} and slightly above the ground state case in \cite{NakSch12} were considered.
To investigate the time behavior of solutions with larger mass-energy, Duyckaerts--Roudenko \cite{DuyRou15} used the following refined Gagliardo-Nirenberg inequality:
If we assume additionally that
$
	xf
		\in L^2(\R^d)
$
holds, then the inequality \eqref{G-N inequality} is improved as
\begin{align}\label{108}
	N(f)
		\leq C_{\rm GN}M(f)^\frac{(p-1)(1-s_c)}{2}\left\{L(f) - \frac{\<ixf,\nabla f\>_{L^2}^2}{\|xf\|_{L^2}^2}\right\}^{\frac{(p-1)s_c}{2}+1}
\end{align}
for any
$
	f
		\in H^1(\R^d)
$
with
$
	xf
		\in L^2(\R^d)
$
by optimizing with
$
	e^{i\lambda |x|^2}f,
$
where
$
	\<\cdot,\cdot\>_{L^2}
$
denotes the
$
	L^2
$-scalar inner product defined as
\begin{align*}
	\<f,g\>_{L^2}
		:= \Re \int_{\R^d}f(x)\overline{g(x)}dx.
\end{align*}
Applying this inequality, Duyckaerts--Roudenko \cite{DuyRou15} showed the following theorem:

\begin{theorem}\label{Thm:DuyRou}
Let
$
	d
		\in \N,
$
$
	1+\frac{4}{d}
		< p
		< 2^\ast-1,
$
$
	u_0
		\in H^1(\R^d),
$
and
$
	xu_0
		\in L^2(\R^d).
$
Assume that
$
	u_0
$
satisfy
\begin{align} \label{103}
	M(u_0)^\frac{1-s_c}{s_c}\left\{E(u_0) - \frac{\<ixu_0,\nabla u_0\>_{L^2}^2}{2\|xu_0\|_{L^2}^2}\right\}
		\leq M(Q_1)^\frac{1-s_c}{s_c}E(Q_1).
\end{align}
\begin{itemize}
\item
(Scattering)
If
$
	M(u_0)^{1-s_c}N(u_0)^{s_c}
		< M(Q_1)^{1-s_c}N(Q_1)^{s_c}
$
and
$
	\<ixu_0,\nabla u_0\>_{L^2}
		\geq 0,
$
then the solution
$
	u
$
to \eqref{NLS} scatters in the forward time and
\begin{align*}
	\limsup_{t \rightarrow \infty}M(u(t))^{1-s_c}N(u(t))^{s_c}
		< M(Q_1)^{1-s_c}N(Q_1)^{s_c}.
\end{align*}
\item
(Blow-up)
If
$
	M(u_0)^{1-s_c}N(u_0)^{s_c}
		> M(Q_1)^{1-s_c}N(Q_1)^{s_c}
$
and
$
	\<ixu_0,\nabla u_0\>_{L^2}
		\leq 0,
$
then the solution
$
	u
$
to \eqref{NLS} blows up in the forward time.
\end{itemize}
\end{theorem}

\begin{remark}
There exists an initial data sequence
$
	\{u_{0,n}\}
$
such that all conditions for scattering in Theorem \ref{Thm:DuyRou} and
$
	M(u_{0,n})^\frac{1-s_c}{s_c}E(u_{0,n})
		\longrightarrow \infty
$
as
$
	n
		\rightarrow \infty.
$
For the blow-up case, we can take a similar sequence.
That is, Theorem \ref{Thm:DuyRou} includes solutions with sufficiently large mass-energy compared with \cite{NakSch12}.
On the other hand, the results can only deal with time behaviors in one time direction.
\end{remark}

Noting that the optimizer to \eqref{108} is
$
	e^{i\lambda |x|^2}Q_\omega
$
$
	(\lambda \in \R),
$
Theorem \ref{Thm:DuyRou} can be rewritten as follows:
\begin{theorem}\label{S vs B}
Let
$
	d
		\in \N,
$
$
	1+\frac{4}{d}
		< p
		< 2^\ast-1,
$
$
	u_0
		\in H^1(\R^d),
$
and
$
	xu_0
		\in L^2(\R^d).
$
Assume that
\begin{align}\label{101}
	\text{there exist }
	(\omega, \lambda)
		\in (0,\infty) \times \R
	\text{ such that }
	S_\omega(e^{-\lambda i|x|^2}u_0)
		\leq S_\omega(Q_\omega).
\end{align}
\begin{itemize}
\item
(Scattering)
If
$
	K(e^{-\lambda i|x|^2}u_0)
		> 0
$
and
$
	\<ixu_0,\nabla u_0\>_{L^2}
		\geq 0,
$
then the solution
$
	u
$
to \eqref{NLS} scatters in the forward time.
\item
(Blow-up)
If
$
	K(e^{-\lambda i|x|^2}u_0)
		< 0
$
and
$
	\<ixu_0,\nabla u_0\>_{L^2}
		\leq 0,
$
then the solution
$
	u
$
to \eqref{NLS} blows up in the forward time.
\end{itemize}
\end{theorem}

\begin{remark}
The framework in Theorem \ref{S vs B} includes an initial data sequence
$
	\{u_{0,n}\}
$
such that all conditions for scattering in the theorem and
$
	S_\omega(u_{0,n})
		\longrightarrow \infty
$
as
$
	n
		\rightarrow \infty.
$
For the blow-up case, we can take a similar sequence.
Indeed, take
$
	v_0^\pm
$
satisfying
\begin{align*}
	S_\omega(v_0^\pm)
		< S_\omega(Q_\omega)
	\ \text{ and }\ 
	\left\{ \hspace{-0.2cm}
	\begin{array}{l}
	K(v_0^+) > 0, \\
	K(v_0^-) < 0.
	\end{array}
	\right.
\end{align*}
For example,
$
	v_0^\pm
		= (1 \mp \e)Q_\omega
$
$
	(0
		< \e
		\ll 1).
$
$
	\lambda^+
$
(resp.
$
	\lambda^-
$)
denotes the positive (resp. negative) unique solution to
\begin{align*}
	S_\omega(e^{\lambda^\pm i|x|^2}v_0^\pm)
		= S_\omega(Q_\omega).
\end{align*}
Then, we set
$
	u_{0,n}^+
		:= e^{\lambda i|x|^2}v_0^+
$
for
$
	\lambda
		= \lambda(n)
		:= n
		\geq \lambda^+
		(> 0)
$
and
$
	u_{0,n}^-
		:= e^{\lambda i|x|^2}v_0^-
$
for
$
	\lambda
		= \lambda(n)
		:= - n
		\leq \lambda^-
		(< 0),
$
which are the desired sequences.
\end{remark}

Actually, Theorem \ref{S vs B} is true and it follows from the fact that each corresponding conditions are equivalent (see Propositions \ref{Equivalent conditions}, \ref{Prop: K, MN, and T} and Remark \ref{Rem: K, MN, and T} below).
Regarding as a below the ground state setting, the equivalences are expected naturally.
Duyckaerts--Roudenko \cite{DuyRou15} used a scattering criterion
\begin{align*}
	\limsup_{t \rightarrow \infty}M(u(t))^{1-s_c}N(u(t))^{s_c}
		< M(Q)^{1-s_c}N(Q)^{s_c}
	\ \Longrightarrow\ 
	u\text{ scatters in the forward time}
\end{align*}
to prove the scattering result.
On the other hand, the condition
$
	K(e^{-\lambda i|x|^2}u(t))
		> 0
$
may be incompatible with a linear profile decomposition (Lemma \ref{LPD}).
Then, we need to rewrite the scattering criterion additionally and achieve it by using a functional
$
	T_\omega(f)
		:= \frac{\omega}{2}M(f) + \frac{s_c(p-1)}{2(p+1)}N(f)
$
(see Theorem \ref{S criterion}).

The sign condition of
$
	\<ixu_0,\nabla u_0\>_{L^2}
$
in Theorem \ref{S vs B} can be replaced with the sign of
$
	\lambda.
$
Namely, the following claim holds:

\begin{corollary}\label{Weaker Th}
Let
$
	d
		\in \N,
$
$
	1+\frac{4}{d}
		< p
		< 2^\ast-1,
$
$
	u_0
		\in H^1(\R^d),
$
and
$
	xu_0
		\in L^2(\R^d).
$
Assume that \eqref{101}.
\begin{itemize}
\item
(Scattering)
If
$
	K(e^{-\lambda i|x|^2}u_0)
		> 0
$
and
$
	\lambda
		> 0,
$
then the solution
$
	u
$
to \eqref{NLS} scatters in the forward time.
\item
(Blow-up)
If
$
	K(e^{-\lambda i|x|^2}u_0)
		< 0
$
and
$
	\lambda
		< 0,
$
then the solution
$
	u
$
to \eqref{NLS} blows up in the forward time.
\end{itemize}
\end{corollary}

In this paper, we prove Theorem \ref{S vs B} by using notations
$
	S_\omega
$
and
$
	K
$
not
$
	ME
$
and
$
	MN.
$
Its aim is not just interesting in an alternative proof, but also application to case without scale invariance such as
\begin{align}\label{NLSP}
	i\partial_t u + \Delta u - Pu
		= - |u|^{p-1}u,
	\qquad (t,x) \in \R \times \R^d
\end{align}
and
\begin{align}\label{NLSdou}
	i\partial_tu + \Delta u
		= - |u|^{q-1}u - |u|^{p-1}u,
	\qquad (t,x) \in \R \times \R^d.
\end{align}
Especially, we envisage
$
	P(x)
		= \frac{\gamma}{|x|^\mu},
$
$
	P(x)
		= \gamma \delta_0(x),
$
and so on.
In the case of equations with scale invariance, there are some results, such as \cite{Ard21, CamCar21, DenLuMen23, DinKer21, GaoWan20, GaoWan21, Saa20, Saa21, SaaFen22, SaaFen24, SaaNaf24, YanLiWuCac17} that use a similar framework to Theorem \ref{Thm:DuyRou}.
We turn to \eqref{NLSP} with
$
	P(x)
		= \frac{\gamma}{|x|^\mu}.
$
For the purpose, we set up a minimization problem
\begin{align*}
	r_{\omega,P}
		:= \inf\{S_{\omega,P}(f) : f \in H_{\rad}^1(\R^d) \setminus \{0\},\ K_P(f) = 0\},
\end{align*}
where
\begin{align*}
	S_{\omega,P}(f)
		:= S_\omega(f) + \frac{1}{2}\int_{\R^d}P(x)|f(x)|^2dx
	\ \text{ and }\ 
	K_P(f)
		:= K(f) - \int_{\R^d}(x\cdot \nabla P)|f(x)|^2dx
\end{align*}
for
$
	S_\omega
$
defined in \eqref{Sw} and
$
	K
$
defined in \eqref{K}.
It is known in \cite[Theorem 1.5]{HamIkeISAAC} that there exists an optimizer
$
	Q_{\omega,P}
$
to the minimization problem
$
	r_{\omega,P}
$
with
$
	d
		\geq 2,
$
$
	1 + \frac{4}{d}
		< p
		< 2^\ast-1,
$
$
	P(x)
		= \frac{\gamma}{|x|^\mu},
$
$
	\gamma
		> 0,
$
and
$	
	0
		< \mu
		< 2.
$
Using the minimizer
$
	Q_{\omega,P},
$
the author and Ikeda \cite{HamIkeJMP, HamIkeJMAA} revealed global dynamics of below the minimizer
$
	Q_{\omega,P}.
$

\begin{theorem}
Let
$
	d
		= p
		= 3,
$
$
	P(x)
		= \frac{\gamma}{|x|^\mu},
$
$
	\gamma
		> 0,
$
$
	0
		< \mu
		< 2,
$
and let a radial data
$
	u_0
		\in H_{\rad}^1(\R^3)
$
satisfy
$
	S_{\omega,P}(u_0)
		< S_{\omega,P}(Q_{\omega,P})
$
for some
$
	\omega
		> 0.
$
Then, the following is true:
\begin{itemize}
\item
(Scattering)
If
$
	K_P(u_0)
		\geq 0,
$
then the solution
$
	u
$
to \eqref{NLSP} scatters.
\item
(Blow-up)
If
$
	K_P(u_0)
		< 0,
$
then the solution
$
	u
$
to \eqref{NLSP} blows up.
\end{itemize}
\end{theorem}

In this paper, we obtain the following theorem by using the proof of Theorem \ref{S vs B}:

\begin{theorem}\label{S vs B P}
Let
$
	d
		= p
		= 3,
$
$
	P(x)
		= \frac{\gamma}{|x|^\mu},
$
$
	\gamma
		> 0,
$
$
	0
		< \mu
		< 2,
$
and let a radial data
$
	u_0
		\in H_{\rad}^1(\R^3)
$
satisfy
$
	xu_0
		\in L^2(\R^3).
$
Assume that there exist
$
	(\omega,\lambda)
		\in (0,\infty) \times \R
$
such that
\begin{align*}
	S_{\omega,P}(e^{-\lambda i|x|^2}u_0)
		\leq S_{\omega,P}(Q_{\omega,P}).
\end{align*}
\begin{itemize}
\item
(Scattering)
If
$
	K_P(e^{-\lambda i|x|^2}u_0)
		> 0
$
and
$
	\<ixu_0,\nabla u_0\>_{L^2}
		\geq 0,
$
then the solution
$
	u
$
to \eqref{NLSP} scatters in the forward time.
\item
(Blow-up)
If
$
	K_P(e^{-\lambda i|x|^2}u_0)
		< 0
$
and
$
	\<ixu_0,\nabla u_0\>_{L^2}
		\leq 0,
$
then the solution
$
	u
$
to \eqref{NLSP} blows up in the forward time.
\end{itemize}
\end{theorem}

\begin{remark}
We note some remarks for the theorem.
\begin{itemize}
\item
(Restriction of
$
	d
$
and
$
	p
$)
We think that
$
	d
$
and
$
	p
$
can be extended to some range.
For the global existence and blow-up results, we can get them under
$
	d
		\in \N,
$
$
	1+\frac{4}{d}
		< p
		< 2^\ast-1,
$
and
$
	0
		< \mu
		< \min\{2,d\}
$
in the same manner as Theorem \ref{S vs B P}.
However, it seems difficult for scattering result that we relax it to full intercritical range
($
	1+\frac{4}{d}
		< p
		< 2^\ast-1
$)
in the absence of a non-admissible Strichartz estimate.
By the reason, we do not aim to generalize
$
	d
$
and
$
	p
$
here.
\item
(Radial assumption)
To analyze non-radial solutions, if we consider a minimization problem
\begin{align*}
	n_{\omega,P}
		:= \inf\{S_{\omega,P}(f) : f \in H^1(\R^d) \setminus \{0\},\ K_P(f) = 0\},
\end{align*}
then
$
	n_{\omega,P}
		= n_\omega
$
holds (see \cite[Theorem 1.5]{HamIkeISAAC}), where
$
	n_\omega
$
is defined in \eqref{nw}.
Therefore, its framework can be made use of
$
	ME
$
and
$
	MN
$
formulations in Theorem \ref{Thm:DuyRou}.
Indeed, the result is seen in \cite{Wan23} (see also \cite{ArdHaj23, DenLuMen231, Din22, Din24, DinForHaj22, GaoWan19, JiLuMen25, YuaFen20, YueLiZha16} for this direction works).
\end{itemize}
\end{remark}

It is expected that the argument for \eqref{NLSP} with
$
	P(x)
		= \frac{\gamma}{|x|^\mu}
$
can be applied to other equations as nonlinear Schr\"odinger equation with a delta potential, that is, \eqref{NLSP} with
$
	P(x)
		= \gamma \delta_0(x),
$
nonlinear Schr\"odinger equation with a double power nonlinearity \eqref{NLSdou}, and so forth.

The rest of the paper is organized as follows:
In Section \ref{Sec: Equivalence of conditions}, we see equivalence on conditions in Theorems \ref{Thm:DuyRou} and \ref{S vs B}.
In Section \ref{Sec: Well-posedness}, we recall a well-posedness of \eqref{NLS}.
In Section \ref{Sec: G vs B}, we prove the global existence and blow-up results in Theorem \ref{S vs B}.
In Section \ref{Sec: S}, we complete the proof of Theorem \ref{S vs B} by showing the scattering criterion (Theorem \ref{S criterion}).
In Section \ref{Sec:Application}, we apply the proof in Sections \ref{Sec: G vs B} and \ref{Sec: S} to the equation with the inverse-power potential.
In Appendix \ref{Alternative}, we give an alternative proof of coercivity for the virial functional
$
	K
$
(Lemma \ref{Coercivity}) by using the scaling property of the equation.

\section{Equivalence of conditions}\label{Sec: Equivalence of conditions}

In this section, we show the equivalence on the conditions in Theorems \ref{Thm:DuyRou} and \ref{S vs B}.
Therefore, we can conclude that Theorem \ref{S vs B} is true.

At first, we check that
$
	e^{\lambda i|x|^2}Q_\omega
$
is a optimizer to
\begin{align*}
	n_{\omega,\lambda}
		:= \inf\{S_\omega(e^{-\lambda i|x|^2}f) : f \in H^1(\R^d) \setminus \{0\},\ xf \in L^2(\R^d),\ K(e^{-\lambda i|x|^2}f) = 0\}
\end{align*}
for
$
	(\omega,\lambda)
		\in (0,\infty) \times \R.
$
We have already known that
$
	f
		= e^{\lambda i|x|^2}Q_\omega
$
attains to
\begin{align*}
	n_{\omega,\lambda}^\ast
		:= \inf\{S_\omega(e^{-\lambda i|x|^2}f) : e^{-\lambda i|x|^2}f \in H^1(\R^d) \setminus \{0\},\ K(e^{-\lambda i|x|^2}f) = 0\}
\end{align*}
and
$
	n_{\omega,\lambda}^\ast
		= S_\omega(Q_\omega)
$
for
$
	(\omega,\lambda)
		\in (0,\infty) \times \R.
$
The inclusion
\begin{align*}
	\{f : f \in H^1(\R^d) \setminus \{0\},\ xf \in L^2(\R^d)\}
		\subset \{f : e^{-\lambda i|x|^2}f \in H^1(\R^d) \setminus \{0\}\}
\end{align*}
gives
$
	n_{\omega,\lambda}^\ast
		\leq n_{\omega,\lambda}.
$
In addition,
$
	e^{\lambda i|x|^2}Q_\omega
		\in H^1(\R^d) \setminus \{0\}
$
and
$
	xe^{\lambda i|x|^2}Q_\omega
		\in L^2(\R^d)
$
implies
$
	n_{\omega,\lambda}
		\leq S_\omega(Q_\omega).
$
We note that
$
	Q_\omega
$
is exponentially decaying (see \cite[Theorem 8.1.1]{Caz03}).
Collecting these, we have
$
	n_{\omega,\lambda}
		= n_{\omega,\lambda}^\ast
		= S_\omega(Q_\omega).
$

Now, we are in a position to check the equivalence to the frameworks in Theorems \ref{Thm:DuyRou} and \ref{S vs B}.

\begin{proposition}\label{Equivalent conditions}
Let
$
	f
		\in H^1(\R^d)
$
satisfy
$
	xf
		\in L^2(\R^d).
$
The condition \eqref{103} is equivalent to the following conditions:
\begin{enumerate}
\item
There exists
$
	\lambda
		\in \R
$
such that
$
	M(f)^\frac{1-s_c}{s_c}E(e^{-\lambda i|x|^2}f)
		\leq M(Q_1)^\frac{1-s_c}{s_c}E(Q_1),	
$
\item
There exists
$
	(\omega,\lambda)
		\in (0,\infty) \times \R
$
such that
$
	S_\omega(e^{-\lambda i|x|^2}f)
		\leq S_\omega(Q_\omega),
$
\item
There exists
$
	\omega
		> 0
$
such that
$\displaystyle
	S_\omega(f) - \frac{\<ixf,\nabla f\>_{L^2}^2}{2\|xf\|_{L^2}^2}
		\leq S_\omega(Q_\omega).
$
\end{enumerate}
\end{proposition}

\begin{proof}
\eqref{103}
$
	\Longleftrightarrow
$
(1) :
Since
$
	N(e^{-\lambda i|x|^2}f)
		= N(f)
$
and
\begin{align*}
	L(e^{-\lambda i|x|^2}f)
		= L(f) + 4\lambda^2\|xf\|_{L^2}^2 - 4\lambda\<ixf,\nabla f\>_{L^2},
\end{align*}
$
	\min_{\lambda \in \R}\{M(f)^\frac{1-s_c}{s_c}E(e^{-\lambda i|x|^2}f)\}
$
has the minimum value
\begin{align*}
	M(f)^\frac{1-s_c}{s_c}\left\{E(f) - \frac{\<ixf,\nabla f\>_{L^2}^2}{2\|xf\|_{L^2}^2}\right\}
\end{align*}
at
$
	\lambda_{\min}
		= \frac{\<ixf,\nabla f\>_{L^2}}{2\|xf\|_{L^2}^2},
$
which implies the desired result.

For (1)
$
	\Longleftrightarrow
$
(2), e.g., see \cite[Lemma A.1]{IkeInu17}.

(2)
$
	\Longleftrightarrow
$
(3) follows from the same manner as \eqref{103}
$
	\Longleftrightarrow
$
(1).
\end{proof}

\begin{remark}
We notice that the condition (3) in Proposition \ref{Equivalent conditions} can be rewritten as ``There exists
$
	\omega
		> 0
$
such that
$\displaystyle
	S_\omega(e^{-\lambda_{\min} i|x|^2}f)
		\leq S_\omega(Q_\omega)"
$
by using
$
	\lambda_{\min}
		= \frac{\<ixf,\nabla f\>_{L^2}}{2\|xf\|_{L^2}^2},
$
\end{remark}

Next, we establish equivalence to
$
	MN
$
condition in Theorem \ref{Thm:DuyRou} and
$
	K
$
condition in Theorem \ref{S vs B}.
In addition, a condition denoted by
$
	T_\omega
$
is also identical to that of the other two.
Here,
$
	T_\omega
$
is defined as
\begin{align}\label{142}
	T_\omega(f)
		:= S_\omega(f) - \frac{1}{4}K(f)
		= \frac{\omega}{2}M(f) + \frac{s_c(p-1)}{2(p+1)}N(f)
\end{align}
and it has a property
\begin{align}\label{136}
	S_\omega(Q_\omega)
		= T_\omega(Q_\omega)
		= \inf\{T_\omega(f) : f \in H^1(\R^d)\setminus\{0\},\ K(f) \leq 0\}
\end{align}
(e.g., see \cite[Lemma 3.1]{HamIkeISAAC}).

\begin{proposition}\label{Prop: K, MN, and T}
Let
$
	f
		\in H^1(\R^d)
$
satisfy
$
	xf
		\in L^2(\R^d).
$
Suppose
$
	S_\omega(e^{-\lambda i|x|^2}f)
		\leq S_\omega(Q_\omega)
$
for some
$
	(\omega, \lambda)
		\in (0,\infty) \times \R.
$
Then, the following three conditions are equivalent.
\begin{enumerate}
\item
$
	K(e^{-\lambda i|x|^2}f)
		< 0,
$
\item
$
	M(f)^\frac{1-s_c}{s_c}N(f)
		> M(Q_1)^\frac{1-s_c}{s_c}N(Q_1),
$
\item
$
	T_\omega(Q_\omega)
		< T_\omega(f).
$
\end{enumerate}
\end{proposition}

\begin{remark}\label{Rem: K, MN, and T}
In addition, if we insert
$
	``="
$
and repeat the same proof as Proposition \ref{Prop: K, MN, and T}, then we get a equivalence of
$
	K(e^{-\lambda i|x|^2}f)
		\leq 0,
$
$
	M(f)^\frac{1-s_c}{s_c}N(f)
		\geq M(Q_1)^\frac{1-s_c}{s_c}N(Q_1),
$
and
$
	T_\omega(Q_\omega)
		\leq T_\omega(f).
$
By taking the contraposition of these equivalences, the following three conditions also are equivalent under the conditions in Proposition \ref{Prop: K, MN, and T}:
\begin{itemize}
\item
$
	K(e^{-\lambda i|x|^2}f)
		> 0,
$
\item
$
	M(f)^\frac{1-s_c}{s_c}N(f)
		< M(Q_1)^\frac{1-s_c}{s_c}N(Q_1),
$
\item
$
	T_\omega(Q_\omega)
		> T_\omega(f),
$
\end{itemize}
and
\begin{itemize}
\item
$
	K(e^{-\lambda i|x|^2}f)
		= 0,
$
\item
$
	M(f)^\frac{1-s_c}{s_c}N(f)
		= M(Q_1)^\frac{1-s_c}{s_c}N(Q_1),
$
\item
$
	T_\omega(Q_\omega)
		= T_\omega(f).
$
\end{itemize}
\end{remark}

\begin{proof}[Proof of Proposition \ref{Prop: K, MN, and T}]
(1)
$
	\Longrightarrow
$
(2):
We note that the condition (1) deduces that
\begin{align*}
	K(e^{-\lambda_{\min} i|x|^2}f)
		= K(f) - \frac{2\<ixf,\nabla f\>_{L^2}^2}{\|xf\|_{L^2}^2}
		< 0
\end{align*}
by the argument in Proposition \ref{Equivalent conditions}.
Applying the refined Gagliardo--Nirenberg inequality \eqref{108}, we have
\begin{align*}
	0
		> K(f) - \frac{2\<ixu_0,\nabla u_0\>_{L^2}^2}{\|xu_0\|_{L^2}^2}
		\geq 2N(f)\left\{\frac{1}{C_{\rm GN}^\frac{2}{(p-1)s_c+2}M(f)^{\frac{2}{d}(1-s_c)}N(f)^{\frac{2}{d}s_c}} - \frac{d(p-1)}{2(p+1)}\right\},
\end{align*}
that is,
\begin{align*}
	M(f)^{\frac{2}{d}(1-s_c)}N(f)^{\frac{2}{d}s_c}
		> \frac{2(p+1)}{d(p-1)}\cdot \frac{1}{C_{\rm GN}^\frac{2}{(p-1)s_c+2}}.
\end{align*}
Therefore, recalling that
$
	Q_1
$
is an optimizer to the Gagliardo--Nirenberg inequality \eqref{G-N inequality} and has \eqref{PI}, we obtain (2).

(2)
$
	\Longrightarrow
$
(3):
If we suppose (2), then it follows from
$
	Q_\omega
		= \omega^\frac{1}{p-1}Q_1(\omega^\frac{1}{2}\cdot),
$
\eqref{PI}, and Young's inequality that
\begin{align*}
	T_\omega(Q_\omega)
		& = \frac{\omega^{1-s_c}}{2(1-s_c)}M(Q_1)^{1-s_c}\cdot \left\{\frac{p-1}{p+1}(1-s_c)\right\}^{s_c}N(Q_1)^{s_c} \\
		& < \left\{\frac{\omega}{2(1-s_c)}M(f)\right\}^{1-s_c}\left\{\frac{p-1}{2(p+1)}N(f)\right\}^{s_c}
		\leq T_\omega(f),
\end{align*}
which implies (3).

(3)
$
	\Longrightarrow
$
(1):
Combined (3) with the condition
$
	S_\omega(e^{-\lambda i|x|^2}f)
		\leq S_\omega(Q_\omega),
$
we have
\begin{align*}
	T_\omega(f) + \frac{1}{4}K(e^{-\lambda i|x|^2}f)
		= S_\omega(e^{-\lambda i|x|^2}f)
		\leq S_\omega(Q_\omega)
		= T_\omega(Q_\omega)
		< T_\omega(f),
\end{align*}
that is, (1) holds.
\end{proof}

\section{Well-posedness}\label{Sec: Well-posedness}

The conditions
$
	d
		\in \N
$
and
$
	1+\frac{4}{d}
		< p
		< 2^\ast-1
$
are assumed throughout Sections \ref{Sec: Well-posedness}-\ref{Sec: S}.

In this section, we recall well-posedness of \eqref{NLS}.
To state this, we put some spaces.
Throughout this paper, we fix a number
$
	r_0
$
with
$
	p+1
		< r_0
		< 2^\ast
$
and set
$
	q_0, q_1, \~{q}_1
$
as
\begin{align*}
	\frac{1}{q_0}
		:= \frac{d}{2}\(\frac{1}{2} - \frac{1}{r_0}\),
	\quad
	\frac{1}{q_1}
		:= \frac{d}{2}\(\frac{1}{2} - \frac{1}{r_0} - \frac{s_c}{d}\),
	\quad
	\frac{1}{\~{q}_1}
		:= \frac{d}{2}\(\frac{1}{2} - \frac{1}{r_0} + \frac{s_c}{d}\).
\end{align*}
We note that
$
	(q_0,r_0),
$
$
	(q_1,r_0),
$
$
	(\~{q}_1,r_0)
$
are admissible pairs of
$
	L^2,
$
$
	\dot{H}^{s_c},
$
$
	\dot{H}^{-s_c}
$
respectively.
Besides these pairs, we define a
$
	\dot{H}^{s_c}
$-admissible pair
$
	(q_2,r_2)
$
as
\begin{align*}
	\frac{p-1}{r_2}
		:= 1 - \frac{2}{r_0},
	\quad
	\frac{1}{q_2}
		:= \frac{d}{2}\(\frac{1}{2} - \frac{1}{r_2} - \frac{s_c}{d}\).
\end{align*}

We define spaces
\begin{align*}
	X(I)
		& := L_t^\infty(I;L_x^2) \cap L_t^{q_0}(I;L_x^{r_0}), \quad
	S(I)
		:= L_t^{q_1}(I;L_x^{r_0}) \cap L_t^{q_2}(I;L_x^{r_2}), \\
	N(I)
		& := {L_t^{\~{q}_1'}(I;L_x^{r_0'})}.
\end{align*}
In addition, we put a space
$
	X^1(I)
$
with a norm
$
	\|f\|_{X^1(I)}
		:= \|(1-\Delta)^\frac{1}{2}f\|_{X(I)}.
$
We note that the solution
$
	u
$
to \eqref{NLS} satisfies
\begin{align}\label{131}
	\|u\|_{S(I)}
		\leq \|e^{it\Delta}u(0)\|_{S(I)} + C\|u^p\|_{N(I)}
		\leq \|e^{it\Delta}u(0)\|_{S(I)} + C\|u\|_{S(I)}^p
\end{align}
by a Strichartz estimate (see \cite[Lemma 3.1]{AkaNaw13}).

Then, we can find the following in \cite[Proposition 3.4]{AkaNaw13}:

\begin{proposition}\label{S-norm}
Assume that the solution
$
	u
$
to \eqref{NLS} exists on
$
	[0,\infty)
$
and has
\begin{align*}
	\|u\|_{S([0,\infty))}
		< \infty,
	\quad
	\|u\|_{L_t^\infty([0,\infty);H_x^1)}
		< \infty.
\end{align*}
Then,
$
	u
$
satisfies
\begin{align*}
	\|u\|_{X^1([0,\infty))}
		< \infty
\end{align*}
and scatters in the forward time.
\end{proposition}

We see small data scattering in \cite[Proposition 3.5]{AkaNaw13}.

\begin{proposition}[Small data scattering]\label{SDS}
Let
$
	t_0
		\in \R
$
and
$
	I
$
be a time interval with
$
	t_0
		\in \overline{I}.
$
Then, there exists
$
	\delta
		> 0
$
such that, for any
$
	u_0
		\in H^1(\R^d)
$
satisfying
\begin{align*}
	\|e^{i(t-t_0)\Delta}u_0\|_{S(I)}
		\leq \delta,
\end{align*}
there exists unique solution
$
	u
		\in C(I;H^1(\R^d))
$
to \eqref{NLS} with
$
	u(t_0)
		= u_0
$
and
\begin{align*}
	\|u\|_{S(I)}
		\leq 2\|e^{i(t-t_0)\Delta}u_0\|_{S(I)}
	\ \text{ and }\ 
	\|u\|_{X^1(I)}
		\lesssim \|u_0\|_{H^1}.
\end{align*}
\end{proposition}

Next, we see long time perturbation theory in \cite[Proposition 3.6]{AkaNaw13}.

\begin{proposition}[Long time perturbation]\label{LTP}
For given
$
	A
		> 1,
$
there exists
$
	\e
		= \e(A)
		> 0
$
such that the following is true:
Let
$
	I
$
be a time interval and
$
	\~{u}
		\in C_t(I;H_x^1(\R^d))
$
be a function satisfying
\begin{align*}
	\|\~{u}\|_{S(I)}
		\leq A
	\ \text{ and }\ 
	\|i\partial_t \~{u} + \Delta \~{u} + |\~{u}|^{p-1}\~{u}\|_{N(I)}
		\leq \e.
\end{align*}
If a function
$
	u_0
		\in H^1(\R^d)
$
satisfies
\begin{align*}
	\|e^{i(t-t_0)\Delta}(u_0 - \~{u}(t_0))\|_{S(I)}
		\leq \e
\end{align*}
for some
$
	t_0
		\in I,
$
then there exists a solution
$
	u
		\in C_t(I;H_x^1(\R^d))
$
to \eqref{NLS} with
$
	u(t_0)
		= u_0
$
such that
\begin{align*}
	\|u\|_{S(I)}
		\lesssim 1.
\end{align*}
\end{proposition}

\section{Global existence and Blow-up}\label{Sec: G vs B}

In this section, we show the global and blow-up results in Theorem \ref{S vs B} by using
$
	S_\omega
$
and
$
	K.
$

\begin{lemma}\label{Estimate of K at t=0}
Let
$
	f
		\in H^1(\R^d)
$
satisfy
$
	xf
		\in L^2(\R^d).
$
If
\eqref{101}
and
$
	K(e^{-\lambda i|x|^2}f)
		< 0
$
holds, then
\begin{align*}
	K(f)
		< 4\{S_\omega(f) - S_\omega(Q_\omega)\}.
\end{align*}
\end{lemma}

\begin{proof}
Applying \cite[Proposition 2.8]{IkeInu17}, we have
\begin{align*}
	K(e^{-\lambda i|x|^2}f)
		< 4\{S_\omega(e^{-\lambda i|x|^2}f) - S_\omega(Q_\omega)\},
\end{align*}
which implies the desired result.
\end{proof}

As a preparation to prove the global and blow-up results, we see also the following lemma.
When
$
	f
$
satisfies \eqref{101} and
$
	K(e^{-\lambda i|x|^2}f)
		< 0,
$
it is obvious that we can take
$
	\lambda
		= \lambda_{\min},
$
where
$
	\lambda_{\min}
$
was defined in the proof of Proposition \ref{Equivalent conditions}.
On the other hand, it is not clear for
$
	f
$
with
$
	K(e^{-\lambda i|x|^2}f)
		> 0.
$
The following lemma implies that we can take
$
	\lambda
		= \lambda_{\min}
$
for
$
	f
$
with
$
	K(e^{-\lambda i|x|^2}f)
		> 0
$
and
$
	S_\omega(e^{-\lambda i|x|^2}f)
		< S_\omega(Q_\omega)
$
instead of
$
	S_\omega(e^{-\lambda i|x|^2}f)
		\leq S_\omega(Q_\omega).
$

\begin{lemma}\label{Persistence of sign}
Let
$
	f
		\in H^1(\R^d)
$
satisfy
$
	xf
		\in L^2(\R^d).
$
Assume that there exists
$
	(\omega, \lambda_0)
		\in (0,\infty) \times \R
$
such that
\begin{align*}
	S_\omega(e^{-\lambda_0 i|x|^2}f)
		< S_\omega(Q_\omega).
\end{align*}
Then, the following is true:
\begin{enumerate}
\item
If
$
	K(e^{-\lambda_0 i|x|^2}f)
		> 0,
$
then
$
	K(e^{-\lambda i|x|^2}f)
		> 0
$
as long as
$
	S_\omega(e^{-\lambda i|x|^2}f)
		< S_\omega(Q_\omega).
$
\item
If
$
	K(e^{-\lambda_0 i|x|^2}f)
		< 0,
$
then
$
	K(e^{-\lambda i|x|^2}f)
		< 0
$
as long as
$
	S_\omega(e^{-\lambda i|x|^2}f)
		< S_\omega(Q_\omega).
$
\end{enumerate}
\end{lemma}

\begin{proof}
We will show only (1) since (2) can be proved in the same manner.
Since
\begin{align*}
	I_\lambda
		:= \{\lambda \in \R : S_\omega(e^{-\lambda i|x|^2}f) < S_\omega(Q_\omega)\}
\end{align*}
is a open interval,
$
	\{K(e^{-\lambda i|x|^2}f) : \lambda \in I_\lambda\}
$
is path-connected.
Therefore, if we deny the desired result, then there exists
$
	\lambda_\ast
		\in I_\lambda
$
such that
$
	K(e^{-\lambda_\ast i|x|^2}f)
		= 0.
$
On the other hand, it follows from the variational characteristic of
$
	Q_\omega
$
\eqref{MP}
that
$
	S_\omega(e^{-\lambda_\ast i|x|^2}f)
		\geq S_\omega(Q_\omega).
$
This is contradiction.
\end{proof}

Next corollary immediately follows from this lemma.

\begin{corollary}
If there exists
$
	(\omega, \lambda)
		\in (0,\infty) \times \R
$
such that
$
	S_\omega(e^{-\lambda i|x|^2}f)
		< S_\omega(Q_\omega)
$
and
$
	K(e^{-\lambda i|x|^2}f)
		> 0.
$
Then, we have
$
	K(e^{-\lambda_{\min} i|x|^2}f)
		> 0,
$
where
$
	\lambda_{\min}
$
was defined in the proof of Proposition \ref{Equivalent conditions}.
\end{corollary}

\begin{proof}
Combining the fact
$
	S_\omega(e^{-\lambda_{\min} i|x|^2}f)
		\leq S_\omega(e^{-\lambda i|x|^2}f)
		< S_\omega(Q_\omega)
$
and Lemma \ref{Persistence of sign}, the corollary follows.
\end{proof}

Here, we prove Corollary \ref{Weaker Th}.
We recall that we have already gotten Theorem \ref{S vs B}.

\begin{proof}[Proof of Corollary \ref{Weaker Th}]
We will show only global existence since the blow-up result is proved by the same argument.
We aim to prove
$
	\<ixu_0,\nabla u_0\>_{L^2}
		\geq 0
$
under the conditions of scattering in Corollary \ref{Weaker Th}.

In the case of
$
	S_\omega(u_0)
		< S_\omega(Q_\omega),
$
the sign of
$
	K(u_0)
$
corresponds to that of
$
	K(e^{-\lambda i|x|^2}u_0).
$
Indeed, if
$
	S_\omega(e^{-\lambda i|x|^2}u_0)
		< S_\omega(Q_\omega)
$
and
$
	K(e^{-\lambda i|x|^2}u_0)
		> 0,
$
then Lemma \ref{Persistence of sign} deduces the desired fact.
On the other hand, when
$
	S_\omega(e^{-\lambda i|x|^2}u_0)
		= S_\omega(Q_\omega)
$
and
$
	K(e^{-\lambda i|x|^2}u_0)
		> 0
$
hold, it can fall down
$
	S_\omega(e^{-\lambda i|x|^2}u_0)
		< S_\omega(Q_\omega)
$
and
$
	K(e^{-\lambda i|x|^2}u_0)
		> 0
$
by shifted
$
	\lambda
$
slightly.
Hence, the desired fact can be gotten in the both cases and Theorem \ref{Below} and Remark \ref{AkaNawCon} determine the time behavior of solutions without regard to the sign of
$
	\lambda.
$

We assume that
$
	S_\omega(u_0)
		\geq S_\omega(Q_\omega).
$
We consider a function
\begin{align*}
	h(\lambda)
		:= S_\omega(e^{-\lambda i|x|^2}u_0)
		= 2\lambda^2\|xu_0\|_{L^2}^2 - 2\lambda\<ixu_0,\nabla u_0\>_{L^2} + S_\omega(u_0).
\end{align*}
Then, we have
$
	h(0)
		\geq S_\omega(Q_\omega)
$
and the set
$
	\mathcal{S}
		:= \{\lambda \in \R : h(\lambda) \leq S_\omega(Q_\omega)\}
$
has a positive element from the assumptions.
Therefore, the axis of
$
	h
$
must be positive, that is,
$
	\<ixu_0,\nabla u_0\>_{L^2}
		> 0.
$
Indeed, if it is non-positive, the set
$
	\mathcal{S}
$
is included in
$
	(-\infty,0],
$
which is contradiction.
\end{proof}

Using these preparations, we show the global existence and blow-up results in Theorem \ref{S vs B}.
We write clearly claim shown  here.

\begin{theorem}\label{G vs B}
Let
$
	u_0
		\in H^1(\R^d)
$
satisfy
$
	xu_0
		\in L^2(\R^d).
$
Assume that
\begin{align}\label{109}
	\text{there exist }
	(\omega, \lambda)
		\in (0,\infty) \times \R
	\text{ such that }
	S_\omega(e^{-\lambda i|x|^2}u_0)
		\leq S_\omega(Q_\omega).
\end{align}
\begin{itemize}
\item
(Global existence)
If
$
	K(e^{-\lambda i|x|^2}u_0)
		> 0
$
and
$
	\<ixu_0,\nabla u_0\>_{L^2}
		\geq 0,
$
then the solution
$
	u
$
to \eqref{NLS} exists globally in the forward time.
\item
(Blow-up)
If
$
	K(e^{-\lambda i|x|^2}u_0)
		< 0
$
and
$
	\<ixu_0,\nabla u_0\>_{L^2}
		\leq 0,
$
then the solution
$
	u
$
to \eqref{NLS} blows up in the forward time.
\end{itemize}
\end{theorem}

\begin{proof}
Let
$
	V(t)
		= \|xu(t)\|_{L^2}^2.
$
Then, we have
\begin{align*}
	V'(t)
		= 4\<ixu(t),\nabla u(t)\>_{L^2}
	\ \text{ and }\ 
	V''(t)
		= 4K(u(t))
\end{align*}
(see \cite[Proposition 6.5.1]{Caz03}).
In addition, we define
$
	z(t)
		= \sqrt{V(t)}.
$
Then, we have
\begin{align*}
	z'(t)
		& = \frac{V'(t)}{2\sqrt{V(t)}}
		= \frac{2\<ixu(t),\nabla u(t)\>_{L^2}}{\|xu(t)\|_{L^2}}, \\
	z''(t)
		& = \frac{1}{z(t)}\[\frac{1}{2}V''(t) - \{z'(t)\}^2\]
		= \frac{2}{\|xu(t)\|_{L^2}}K(e^{-\lambda_{\min}(t) i|x|^2}u(t)),
\end{align*}
where
$
	\lambda_{\min}(t)
		:= \frac{\<ixu(t),\nabla u(t)\>_{L^2}}{2\|xu(t)\|_{L^2}^2}.
$
We recall that
$
	\lambda_{\min}(t)
$
is the minimum point of
\begin{align*}
	\min_{\lambda \in \R}L(e^{-\lambda i|x|^2}u(t)).
\end{align*}
By the same reason in the proof of Corollary \ref{Weaker Th}, we assume that
$
	S_\omega(u_0)
		\geq S_\omega(Q_\omega)
$
in the proof hereafter.
Since
\begin{align*}
	S_\omega(e^{-\lambda_{\min}(0) i|x|^2}u_0)
		= S_\omega(u_0) - \frac{1}{8}\{z'(0)\}^2,
\end{align*}
the condition \eqref{109} is rewritten as
\begin{align}\label{110}
	\{z'(0)\}^2
		\geq 8\{S_\omega(u_0) - S_\omega(Q_\omega)\}
\end{align}
through Proposition \ref{Equivalent conditions}.

(Blow-up)
\eqref{110} and
$
	z'(0)
		\leq 0
$
imply
\begin{align*}
	z'(0)
		\leq - \sqrt{8\{S_\omega(u_0) - S_\omega(Q_\omega)\}}
		\leq 0.
\end{align*}
In addition, it follows from Lemma \ref{Estimate of K at t=0} that
\begin{align*}
	K(u_0)
		< 4\{S_\omega(u_0) - S_\omega(Q_\omega)\}.
\end{align*}
Combining these inequalities, we see
\begin{align*}
	z''(0)
		< \frac{1}{z(0)}\[8\{S_\omega(u_0) - S_\omega(Q_\omega)\} - 8\{S_\omega(u_0) - S_\omega(Q_\omega)\}\]
		= 0.
\end{align*}
Here, we will prove that
$
	z''(t)
		< 0
$
for any
$
	t
		\in [0,T_{\max})
$
by the contradiction.
If not, then there exists
$
	t_0
		\in (0,T_{\max})
$
such that
$
	z''(t)
		< 0
$
for any
$
	t
		\in [0,t_0)
$
and
$
	z''(t_0)
		= 0
$
(that is,
$
	K(e^{-\lambda_{\min}(t_0) i|x|^2}u(t_0))
		= 0
$).
Combining
$
	z'(0)
		\leq 0
$
and
$
	z''(t)
		< 0
$
on
$
	t
		\in [0,t_0),
$
we have
$
	- \{z'(t)\}^2
$
is decreasing strictly on
$
	t
		\in [0,t_0].
$
That is,
\begin{align*}
	S_\omega(e^{-\lambda_{\min}(t) i|x|^2}u(t))
		< S_\omega(e^{-\lambda_{\min}(0) i|x|^2}u_0)
		\leq S_\omega(Q_\omega)
\end{align*}
holds on
$
	t
		\in (0,t_0]
$
and hence, we obtain
\begin{align}\label{190}
	K(e^{-\lambda_{\min}(t) i|x|^2}u(t))
		< 0
\end{align}
on
$
	t
		\in [0,t_0]
$
by the variational characterization of
$
	Q_\omega.
$
Indeed, if not, there exists
$
	t_1
		\in (0,t_0)
$
such that
\begin{align*}
	K(e^{-\lambda_{\min}(t_1) i|x|^2}u(t_1))
		= 0.
\end{align*}
Hence, we have
\begin{align*}
	S_\omega(Q_\omega)
		\leq S_\omega(e^{-\lambda_{\min}(t_1) i|x|^2}u(t_1))
\end{align*}
which is contradiction and we get \eqref{190}.
However, \eqref{190} contradicts the taking of
$
	t_0.
$
Therefore, it is proved that
$
	z''(t)
		< 0
$
on
$
	t
		\in [0,T_{\max}).
$

Finally, we suppose
$
	T_{\max}
		= \infty
$
to complete the proof of the blow-up.
$
	z''(t)
		< 0
$
implies that
$
	z'(t)
		< z'(1)
		< 0
$
on
$
	t
		\in (1,\infty),
$
which contradicts
$
	z(t)
		\geq 0.
$

(Global existence)
We claim that
$
	z''(0)
		> 0
$
and will prove it by contradiction.
If
$
	z''(0)
		= 0, 
$
then
$
	K(e^{-\lambda_{\min}(0) i|x|^2}u_0)
		= 0
$
and hence,
\begin{align}\label{113}
	S_\omega(Q_\omega)
		\leq S_\omega(e^{-\lambda_{\min}(0) i|x|^2}u_0)
		< S_\omega(e^{-\lambda i|x|^2}u_0)
	\text{ for any }
	\lambda
		\in \R \setminus \{\lambda_{\min}(0)\},
\end{align}
which contradicts the assumptions of the theorem.
On the other hand, if
$
	z''(0)
		< 0, 
$
then
$
	K(e^{-\lambda_{\min}(0) i|x|^2}u_0)
		< 0.
$
The case of
$
	S_\omega(e^{-\lambda_{\min}(0) i|x|^2}u_0)
		< S_\omega(Q_\omega)
$
is incorrect from Lemma \ref{Persistence of sign}.
The case of
$
	S_\omega(Q_\omega)
		\leq S_\omega(e^{-\lambda_{\min}(0) i|x|^2}u_0)
$
does not happen by the same reason as \eqref{113}.
In addition,
$
	z''(t)
		> 0
$
is extended to
$
	[0,T_{\max})
$
by a similar argument to the blow-up case.

$
	z''(t)
		> 0
$
on
$
	t
		\in [0,T_{\max})
$
implies
$
	K(u(t))
		> \frac{1}{2}\{z(t)\}^2
		\geq 0.
$
That is, we have
\begin{align*}
	\frac{1}{2}L(u(t))
		& < E(u_0) + \frac{1}{p+1}N(u(t)) + \frac{1}{d(p-1)}K(u(t))
		= E(u_0) + \frac{2}{d(p-1)}L(u(t))
\end{align*}
and hence,
\begin{align}\label{181}
	\left\{\frac{1}{2} - \frac{2}{d(p-1)}\right\}L(u(t))
		< E(u_0)
\end{align}
for any
$
	t
		\in [0,T_{\max}),
$
which completes the proof of the global existence.
\end{proof}

\begin{remark}\label{Uniformly bounded}
The global solutions in Theorem \ref{S vs B} satisfies
\begin{align*}
	\sup_{t \in [t_0,\infty)}T_\omega(u(t))
		< S_\omega(Q_\omega)
\end{align*}
for any
$
	t_0
		\in (0,\infty),
$
where
$
	T_\omega
$
is defined in \eqref{142}.
Indeed, it follows from
$
	z''(t)
		> 0
$
for any
$
	t
		\in [0,\infty)
$
that
$
	S_\omega(e^{-\lambda_{\min}(t) i|x|^2}u(t))
$
is decreasing strictly on
$
	[0,\infty)
$
and hence, we have
\begin{align*}
	\sup_{t \in [t_0,\infty)}T_\omega(u(t))
		\leq \sup_{t \in [t_0,\infty)}S_\omega(e^{-\lambda_{\min}(t) i|x|^2}u(t))
		= S_\omega(e^{-\lambda_{\min}(t_0) i|x|^2}u(t_0))
		< S_\omega(Q_\omega).
\end{align*}
\end{remark}

\section{Scattering}\label{Sec: S}

In this section, we show the scattering result in Theorem \ref{S vs B}.
Under the conditions of scattering in Theorem \ref{S vs B}, we have already shown the global existence satisfying the inequality in Remark \ref{Uniformly bounded}.
Therefore, if we prove the following scattering criterion, then we complete the proof of Theorem \ref{S vs B}.

\begin{theorem}[Scattering criterion]\label{S criterion}
Let
$
	u
		\in C_t([0,\infty);H_x^1(\R^d))
$
be a forward time global solution to \eqref{NLS}.
If there exists
$
	\omega
		> 0
$
such that
\begin{align*}
	\sup_{t \geq 0}T_\omega(u(t))
		< S_\omega(Q_\omega),
\end{align*}
then
$
	u
$
scatters in the forward time.
\end{theorem}

We define
\begin{align*}
	S_\omega^c(A,B)
		:=
		\sup\left\{
		\|u\|_{S([0,\infty))}
		\left|
		\begin{array}{l}
		\text{The solution }
		u
			\in C_t([0,\infty);H_x^1(\R^d))
		\text{ satisfies } \\
		S_\omega(u_0)
			\leq A
		\text{ and }
		\sup_{t \in [0,\infty)}T_\omega(u(t))
			\leq B.
		\end{array}
		\right.
		\right\}
\end{align*}
for
$
	A, B
		> 0.
$
Then, we would like to show the following:
\begin{theorem}\label{Scattering criterion}
Suppose that
$
	A
		> 0
$
and
$
	0
		< B
		< S_\omega(Q_\omega).
$
Then, we have
$
	S_\omega^c(A,B)
		< \infty.
$
\end{theorem}

At first, we note that if
$
	0
		< A, B
		< S_\omega(Q_\omega),
$
then
$
	S_\omega^c(A,B)
		< \infty.
$
Indeed,
$
	S_\omega(u_0)
		< S_\omega(Q_\omega)
$
and
$
	K(u_0)
		\geq 0
$
holds in this case.
If
$
	K(u_0)
		< 0,
$
then
\begin{align*}
	S_\omega(Q_\omega)
		> T_\omega(u(t))
		> \frac{\omega}{2}M(u_0) + \frac{s_c}{d}L(u(t)),
\end{align*}
that is,
$
	\sup_{t \in [0,\infty)}\|u(t)\|_{H^1}
		< \infty.
$
On the other hand, Theorem \ref{Below} and Remark \ref{AkaNawCon} imply that
$
	u
$
blows up or grows up, which is contradiction.

To prove Theorem \ref{Scattering criterion}, we assume that
$
	S_\omega^c(A_\ast,B_\ast)
		= \infty
$
for some
$
	(A_\ast,B_\ast)
		\in (0,\infty) \times (0,S_\omega(Q_\omega))
$
and deduce contradiction.
Then,
\begin{align*}
	S_\omega^c
		= S_\omega^c(B)
		:= \inf\{A > 0 : S_\omega^c(A,B) = \infty\}
\end{align*}
is well-defined for
$
	[B_\ast,S_\omega(Q_\omega)).
$
By the above argument, we can see
$
	S_\omega^c
		\geq S_\omega(Q_\omega)
		> 0.
$
We note also that when
$
	S_\omega^c(B)
		< \infty,
$
the conditions
$
	S_\omega(u_0)
		< S_\omega^c(B)
$
and
$
	\sup_{t \in [0,\infty)}T_\omega(u(t))
		\leq B
$
imply that
$
	u
$
scatters in the forward time by the definition of
$
	S_\omega^c.
$

\subsection{Existence of a critical solution}

In this subsection, we construct the critical solution
$
	u_c,
$
which is a forward non-scattering solution having
$
	S_\omega(u_c)
		= S_\omega^c(B)
$
and
$
	\sup_{t \in [0,\infty)}T_\omega(u_c(t))
		\leq B.
$

\begin{lemma}[Coercivity]\label{Coercivity}
Let
$
	0
		< B
		< S_\omega(Q_\omega).
$
There exists
$
	C
		> 0
$
such that, for any
$
	f
		\in H^1(\R^d)
$
with
$
	T_\omega(f)
		\leq B,
$
we have
\begin{align*}
	K(f)
		\geq CL(f)
	\ \text{ and }\ 
	E(f)
		\geq \frac{C}{4}L(f).
\end{align*}
\end{lemma}

\begin{proof}
The second conclusion follows from the first one.
Then, we see only the first inequality.
Applying the Gagliardo--Nirenberg inequality \eqref{G-N inequality}, we have
\begin{align*}
	K(f)
		& \geq 2L(f) - \frac{d(p-1)}{p+1}C_{\rm GN}M(f)^\frac{(p-1)(1-s_c)}{2}L(f)^{\frac{(p-1)s_c}{2}+1} \\
		& \geq \left\{2 - \frac{d(p-1)}{p+1}C_{\rm GN}\(\frac{2B}{\omega}\)^\frac{(p-1)(1-s_c)}{2}L(f)^\frac{(p-1)s_c}{2}\right\}L(f).
\end{align*}
Therefore, if
$
	L(f)
		\leq C_{\omega,B}
$
for
\begin{align*}
	C_{\omega,B}^{-\frac{(p-1)s_c}{2}}
		:= \frac{d(p-1)}{p+1}C_{\rm GN}\(\frac{2B}{\omega}\)^\frac{(p-1)(1-s_c)}{2},
\end{align*}
then
$
	K(f)
		\geq L(f).
$

On the other hand, if we deny the desired result when
$
	L(f)
		> C_{\omega,B},
$
there exists
$
	\{f_n\}
		\subset H^1(\R^d)
$
such that
$
	L(f_n)
		> C_{\omega,B},
$
$
	T_\omega(f_n)
		\leq B
		< S_\omega(Q_\omega),
$
and
\begin{align}\label{117}
	K(f_n)
		< \frac{1}{n}L(f_n).
\end{align}
\eqref{117} implies that
\begin{align}\label{180}
	\(2 - \frac{1}{n}\)L(f_n)
		< \frac{d(p-1)}{p+1}N(f_n).
\end{align}
Combining \eqref{117} and \eqref{180}, we notice that
$
	\sup_{n \in \N}\|f_n\|_{H^1}
		< \infty
$
holds and hence, we have
$
	K(f_n)
		\longrightarrow 0
$
as
$
	n
		\rightarrow \infty.
$
Here, we use Lieb's compactness lemma \cite[Lemma 6]{Lie83} to get a subsequence of
$
	f_n
$
(which is denoted by the same symbol), a sequence
$
	\{y_n\}
		\subset \R^d,
$
and a non-zero function
$
	g
		\in H^1(\R^d) \setminus\{0\}
$
satisfying
$
	g_n
		:= f_n(\cdot - y_n)
		\rightharpoonup g
$
in
$
	H^1(\R^d).
$
Passing to a subsequence again and applying Brezis--Lieb's lemma \cite{BreLie83}, we get
\begin{align*}
	& K(f_n) - K(g_n - g)
		\longrightarrow K(g), \\
	& T_\omega(f_n) - T_\omega(g_n - g)
		\longrightarrow T_\omega(g)
\end{align*}
as
$
	n
		\rightarrow \infty.
$
We will get contradiction by splitting into
$
	K(g)
		\leq 0
$
and
$
	K(g)
		> 0.
$
Using \eqref{136},
$
	K(g)
		\leq 0
$
deduces that
$
	S_\omega(Q_\omega)
		\leq T_\omega(g).
$
Therefore, we obtain
\begin{align*}
	0
		\leq \liminf_{n \rightarrow \infty}T_\omega(g_n - g)
		\leq \limsup_{n \rightarrow \infty}T_\omega(f_n) - T_\omega(g)
		\leq B - S_\omega(Q_\omega)
		< 0,
\end{align*}
which is contradiction.
On the other hand,
$
	K(g)
		> 0
$
implies that
$
	K(g_n - g)
		< 0
$
for sufficiently large
$
	n,
$
where we recall
$
	K(f_n)
		\longrightarrow 0
$
as
$
	n
		\rightarrow \infty.
$
So, we have
$
	T_\omega(g_n - g)
		\geq S_\omega(Q_\omega)
$
and thus, we obtain
\begin{align*}
	0
		& < T_\omega(g)
		= \lim_{n \rightarrow \infty}\{T_\omega(f_n) - T_\omega(g_n - g)\} \\
		& \leq \limsup_{n \rightarrow \infty}T_\omega(f_n) - \liminf_{n \rightarrow \infty}T_\omega(g_n - g)
		\leq B - S_\omega(Q_\omega)
		< 0,
\end{align*}
which is contradiction.
\end{proof}

To get the critical solution, we use the next linear profile decomposition \cite[Lemma 2.1]{DuyHolRou08}:

\begin{lemma}[Linear profile decomposition]\label{LPD}
Let
$
	\{f_n\}
$
be a bounded sequence in
$
	H^1(\R^d).
$
Passing to a subsequence, there exist
$
	J^\ast
		\in \{0,1,\ldots,\infty\},
$
profiles
$
	\{f^j\}
		\subset H^1(\R^d)
$
satisfying
$
	f^0
		\equiv 0,
$
$
	f^j
		\not\equiv 0
$
for any
$
	1
		\leq j
		\leq J^\ast
$
and
$
	f^j
		\equiv 0
$
for any
$
	j
		\geq J^\ast + 1,
$
parameters
$
	\{(t_n^j,x_n^j)\}
		\subset \R \times \R^d,
$
and reminders
$
	\{R_n^J\}
		\subset H^1(\R^d)
$
such that
\begin{align*}
	f_n(x)
		= \sum_{j=0}^Je^{-it_n^j\Delta}f^j(x-x_n^j) + R_n^J(x)
\end{align*}
for any
$
	0
		\leq J
		< \infty
$
and any
$
	n
		\geq 1.
$
The parameters
$
	\{(t_n^j,x_n^j)\}
$
satisfies
\begin{align*}
	\lim_{n \rightarrow \infty}(|t_n^j - t_n^k| + |x_n^j - x_n^k|)
		= \infty
\end{align*}
for any
$
	j
		\neq k
$
and
\begin{align*}
	& \text{either }t_n^j
		\longrightarrow \pm\infty\ \text{ as }\ n \rightarrow \infty
	\ \ \text{ or }\ \ 
	t_n^j \equiv 0\ \text{ for each }\ n \in \N, \\
	& \text{either }|x_n^j|
		\longrightarrow \infty\ \text{ as }\ n \rightarrow \infty
	\ \ \text{ or }\ \ 
	x_n^j \equiv 0\ \text{ for each }\ n \in \N
\end{align*}
for any
$
	j
		\geq 0.
$
Moreover, $\{R_n^J\}$ satisfies
\begin{align*}
	e^{it_n^j\Delta}R_n^J
		\xrightharpoonup[]{\hspace{0.4cm}}
		\begin{cases}
		&\hspace{-0.4cm}\displaystyle{
			f^j, \quad (J < j),
		}\\
		&\hspace{-0.4cm}\displaystyle{
			\hspace{0.1cm}0,\hspace{0.1cm} \quad (J \geq j)
		}
		\end{cases}
\end{align*}
in
$
	H^1(\R^d)
$
as
$
	n
		\rightarrow \infty
$
for any
$
	j
		\geq 0
$
and
\begin{align}
	\lim_{J \rightarrow \infty}\limsup_{n \rightarrow \infty}\|e^{it\Delta}R_n^J\|_{L_t^\infty(\R;L_x^{\frac{(p-1)d}{2}})\cap S(\R)}
		= 0. \label{118}
\end{align}
Furthermore, Pythagorean decomposition
\begin{align}
	\|f_n\|_{\dot{H}^s}^2
		& = \sum_{j = 0}^J\|f^j\|_{\dot{H}^s}^2 + \|R_n^J\|_{\dot{H}^s}^2 + o_n(1)\ \text{ for }\ 0 \leq s \leq 1, \label{119} \\
	\|f_n\|_{L^{p+1}}^{p+1}
		& = \sum_{j = 0}^J\|e^{-it_n^j\Delta}f^j\|_{L^{p+1}}^{p+1} + \|R_n^J\|_{L^{p+1}}^{p+1} + o_n(1) \label{120}
\end{align}
hold.
\end{lemma}

\begin{theorem}[Existence of the critical solution]\label{Existence of CS}
Let
$
	B
		\in (0,S_\omega(Q_\omega))
$
satisfy
$
	S_\omega^c(B)
		< \infty.
$
Then, there exits a solution
$
	u_c
$
to \eqref{NLS} such that
$
	S_\omega(u_c)
		= S_\omega^c(B),
$
$
	\sup_{t \in [0,\infty)}T_\omega(u_c(t))
		\leq B,
$
and
$
	u_c
$
does not scatter in the forward time.
\end{theorem}

\begin{proof}
Take a sequence
$
	u_{n,0}
		\in H^1(\R^d)
$
satisfying
\begin{itemize}
\item
$
	S_\omega^c(B) + 1
		> S_\omega(u_{n,0})
		\searrow S_\omega^c(B)
$
as
$
	n
		\rightarrow \infty,
$
\item
$
	\sup_{t \in [0,\infty)}T_\omega(u_n(t))
		\leq B
$
for each
$
	n
		\in \N,
$
\item
$
	u_n
$
does not scatters in the forward time for each
$
	n
		\in \N,
$
\end{itemize}
where
$
	u_n
$
denotes the solution to \eqref{NLS} with
$
	u_n(0)
		= u_{n,0}.
$
Then, we have
\begin{align*}
	\omega M(u_n) +	L(u_n(t))
		& = \omega M(u_{n,0}) +	2E(u_{n,0}) + \frac{2}{p+1}N(u_n(t)) \\
		& \leq 2\{S_\omega^c(B) + 1\} + \frac{4}{s_c(p-1)}B
		< \infty,
\end{align*}
that is,
$
	\sup_{n \in \N}\sup_{t \in [0,\infty)}\|u_n(t)\|_{H^1}^2
		\leq C(B,S_\omega^c(B))
		< \infty.
$
Applying the linear profile decomposition (Lemma \ref{LPD}) to
$
	\{u_{n,0}\},
$
we have
\begin{align*}
	u_{n,0}(x)
		= \sum_{j=0}^Je^{-it_n^j\Delta}u^j(x-x_n^j) + U_n^J(x).
\end{align*}

\textbf{(Step 1).}
Exclusion of
$
	J^\ast
		= 0.
$
In this case, we have
\begin{align*}
	\limsup_{n \rightarrow \infty}\|e^{it\Delta} u_{n,0}\|_{S(\R)}
		= \limsup_{n \rightarrow \infty}\|e^{it\Delta} U_n^J\|_{S(\R)}
		= 0,
\end{align*}
which implies
$
	\|u_n\|_{S(\R)}
		< \infty
$
for sufficiently large
$
	n
$
from the small data theory (Propositions \ref{S-norm} and \ref{SDS}).
However, this contradicts taking of
$
	u_{n,0}.
$

\textbf{(Step 2).}
Exclusion of
$
	J^\ast
		\geq 2.
$
Take solutions
$
	v^j
$
to \eqref{NLS} satisfying
\begin{itemize}
\item
$
	v^j(0,x)
		= u^j(x)
$
when
$
	t_n^j
		\equiv 0,
$
\item
$
	\|v^j(-t_n^j,\cdot) - e^{-it_n^j\Delta}u^j\|_{H^1}
		\longrightarrow 0
$
as
$
	n
		\rightarrow \infty
$
when
$
	|t_n^j|
		\longrightarrow \infty
$
as
$
	n
		\rightarrow \infty.
$
\end{itemize}
We note that
\begin{align}\label{129}
	\lim_{n \rightarrow \infty}\|v^j(-t_n^j,\cdot) - e^{-it_n^j\Delta}u^j\|_{H^1}
		\longrightarrow 0
\end{align}
as
$
	n
		\rightarrow \infty
$
in the both cases.
Using
$
	v^j,
$
we write
\begin{align}\label{121}
	u_{n,0}(x)
		= \sum_{j=0}^Jv^j(-t_n^j,x-x_n^j) + V_{n,0}^J(x).
\end{align}
We notice that
\begin{align*}
	V_{n,0}^J(x)
		= \sum_{j=0}^J\{e^{-it_n^j\Delta}u^j(x-x_n^j) - v^j(-t_n^j,x-x_n^j)\} + U_n^J(x)
\end{align*}
and
\begin{align}\label{130}
	\lim_{J\rightarrow \infty}\limsup_{n \rightarrow \infty}\|e^{it\Delta}V_{n,0}^J\|_{S(\R)}
		= 0
\end{align}
by \eqref{118} and \eqref{129}.
It follows from \eqref{119} that
\begin{align}\label{132}
	M(u_{n,0})
		= \sum_{j = 0}^JM(u^j) + M(U_n^J) + o_n(1)
\end{align}
and hence,
\begin{align}\label{126}
	M(v^j(t - t_n^j))
		= M(u^j)
		< M(u_{n,0})
\end{align}
for each
$
	1
		\leq j
		\leq J
$
and sufficiently large
$
	n
		\in \N.
$

Moreover, we use the following decomposition by Holmer--Roudenko \cite[Lemma 6.3]{HolRou10} (see also \cite[Lemma 3.9]{Gue14}):
\begin{lemma}
Let
$
	T
		\in (0,\infty)
$
be a fixed time.
Assume that
$
	u_n
$
exists on
$
	[0,T]
$
for each
$
	n
		\in \N
$
and
\begin{align*}
	\lim_{n \rightarrow \infty}\sup_{t \in [0,T]}L(u_n(t))
		< \infty.
\end{align*}
Then, for the decomposition \eqref{121}, the solution
$
	V_n^J
$
to \eqref{NLS} with
$
	V_n^J(0)
		= V_{n,0}^J
$
and all
$
	v^j
$
$
	(1 \leq j\leq J)
$
exist on
$
	[0,T]
$
and
\begin{align}
	L(u_n(t))
		& = \sum_{j = 0}^JL(v^j(t-t_n^j)) + L(V_n^J(t)) + o_n(1), \label{127} \\
	N(u_n(t))
		& = \sum_{j = 0}^JN(v^j(t-t_n^j)) + N(V_n^J(t)) + o_n(1) \label{128}
\end{align}
for any
$
	t
		\in [0,T],
$
where
$
	o_n(1)
$
goes to
$
	0
$
as
$
	n
		\rightarrow \infty
$
uniformly on
$
	0
		\leq t
		\leq T.
$
\end{lemma}
Then, we can see that
$
	v^j(t-t_n^j)
$
and
$
	V_n^J(t)
$
exist on
$
	[0,\infty)
$
for sufficiently large
$
	n.
$
Especially,
$
	L(v^j(t-t_n^j))
$
and
$
	L(V_n^J(t))
$
are bounded uniformly on
$
	t
		\in [0,\infty)
$
for sufficiently large
$
	n.
$
Combining \eqref{126} and \eqref{128}, we have
\begin{align*}
	\sup_{t \in [0,\infty)}T_\omega(v^j(t-t_n^j))
		< B.
\end{align*}
Applying \eqref{126}, \eqref{127}, \eqref{128}, and the coercivity of the energy (Lemma \ref{Coercivity}), we have
\begin{align*}
	S_\omega(v^j(t-t_n^j))
		< S_\omega^c(B).
\end{align*}
Therefore, we obtain
\begin{align*}
	\|v^j(t-t_n^j)\|_{S([0,\infty))}
		< \infty
\end{align*}
for any
$
	1
		\leq j
		\leq J
$
by the definition of
$
	S_\omega^c(B).
$
Moreover, it follows from \eqref{130} and \eqref{131} that
\begin{align*}
	\lim_{J\rightarrow \infty}\limsup_{n\rightarrow \infty}\|V_n^J\|_{S([0,\infty))}
		= 0.
\end{align*}
Here, we define
\begin{align*}
	u_n^{\leq J}(t,x)
		:= \sum_{j=1}^Jv_n^j(t,x) + e^{it\Delta}U_n^J(x)
\end{align*}
for
$
	v_n^j(t,x)
		:= v^j(t-t_n^j,x-x_n^j)
$
and it follows from an argument in \cite{AkaNaw13} that this is an approximate solution of
$
	u_n
$
in the sense of Proposition \ref{LTP}.
Therefore,
$
	J^\ast
		\geq 2
$
does not hold.

\textbf{(Step 3). Existence of the critical solution}

We have already seen that only one nonzero profile exists.
Then, the profile decomposition can be written as
\begin{align*}
	u_{n,0}(x)
		= e^{-it_n^1\Delta}u^1(x-x_n^1) + U_n(x)
		= v^1(-t_n^1,x-x_n^1) + V_{n,0}(x).
\end{align*}
We show that
$
	t_n^1
		\equiv 0
$
for any
$
	n
		\in \N
$
or
$
	t_n^1
		\longrightarrow \infty
$
as
$
	n
		\rightarrow \infty
$
by contradiction.
If
$
	t_n^1
		\longrightarrow - \infty
$
as
$
	n
		\rightarrow \infty,
$
then
\begin{align*}
	\|e^{it\Delta}u_{n,0}\|_{S(0,\infty)}
		\leq \|e^{it\Delta}u^1\|_{S(-t_n^1,\infty)} + \|e^{it\Delta}U_n\|_{S(0,\infty)}
		\longrightarrow 0
\end{align*}
as
$
	n
		\rightarrow \infty.
$
Combining this limit with Proposition \ref{SDS}, we have
$
	\sup_{n \geq n_0}\|u_n\|_{S(0,\infty)}
		< \infty
$
for sufficiently large
$
	n_0
		\in \N.
$
This is contradiction.
We define
$
	u_c(t,x)
		:= v^1(t,x)
$
and
$
	u_{c,n}(t,x)
		:= u_c(t-t_n^1,x-x_n^1).
$
We see
\begin{align}\label{133}
	\limsup_{n \rightarrow \infty}\|u_{c,n}\|_{S(0,\infty)}
		= \limsup_{n \rightarrow \infty}\|u_c\|_{S(-t_n^1,\infty)}
		= \infty
\end{align}
by the contradiction.
Indeed, if not, then Proposition \ref{LTP} implies that
\begin{align*}
	\|u_n\|_{S(0,\infty)}
		< \infty
\end{align*}
for sufficiently large
$
	n
		\in \N.
$
However, this is contradiction.
Here, we claim that
\begin{align}\label{135}
	\|u_c\|_{S(0,\infty)}
		= \infty.
\end{align}
Since it is the direct conclusion of \eqref{133} in the case of
$
	t_n^1
		\equiv 0
$
for any
$
	n
		\in \N,
$
it suffices to consider only the case of
$
	t_n^1
		\longrightarrow \infty
$
as
$
	n
		\rightarrow \infty.
$
Then, we have
\begin{align*}
	\|e^{it\Delta}u_{n,0}\|_{S(-\infty,0)}
		\leq \|e^{it\Delta}u^1\|_{S(-\infty,-t_n^1)} + \|e^{it\Delta}U_n\|_{S(-\infty,0)}
		\longrightarrow 0
\end{align*}
as
$
	n
		\rightarrow \infty.
$
As the above argument, we have
$
	\sup_{n \geq n_0}\|u_n\|_{S(-\infty,0)}
		< \infty
$
for sufficiently large
$
	n_0
		\in \N.
$
In addition, we have
\begin{align*}
	\sup_{n \geq n_1}\|u_{c,n}\|_{S(-\infty,0)}
		= \sup_{n \geq n_1}\|u_c\|_{S(-\infty,-t_n^1)}
		\leq C
		< \infty
\end{align*}
for sufficiently large
$
	n_1
		\geq n_0
$
from Proposition \ref{LTP} with an approximate solution
$
	u_n(t,x) - e^{it\Delta}U_n^1(x).
$
Noting
$
	\|u_c\|_{S(-t_n^1,0)}
		\leq C(n)
		< \infty,
$
we have
$
	\|u_c\|_{S(-\infty,0)}
		\leq C + C(n)
		< \infty.
$
Combining this boundedness with \eqref{133}, we obtain
$
	\|u_c\|_{S(0,\infty)}
		= \infty.
$
\eqref{132} and \eqref{128} deduces that
\begin{align*}
	\sup_{t \in [0,\infty)}T_\omega(u_c)
		\leq \sup_{t \in [-t_n^1,\infty)}T_\omega(u_c)
		= \sup_{t \in [0,\infty)}T_\omega(u_{c,n})
		\leq B.
\end{align*}
Combining this inequality with \eqref{135}, we have
$
	S_\omega(u_c)
		= S_\omega^c(B).
$
\end{proof}

\subsection{Exclusion of the critical solution}

In this subsection, to get contradiction and complete the proof, we eliminate the critical solution constructed in Theorem \ref{Existence of CS}.

We define a equivalence relation
$
	\sim
$
on
$
	H^1(\R^d)
$
as follows:
\begin{align*}
	{}^\exists x_0 \in \R^d
	\ \text{ s.t. }\ 
	f_1
		= f_2(\,\cdot\, - x_0)
	\ \Longrightarrow\ 
	f_1
		\sim f_2.
\end{align*}
$
	H^1/{\sim}
$
denotes the quotient space, which is constructed by the whole equivalence class with respect to
$
	\sim.
$
We represent an element of
$
	H^1/{\sim}
$
by
$
	[f],
$
and let
$
	\pi : H^1 \longrightarrow H^1/{\sim}
$
be the natural projection.
In addition, this space is complete with respect to a distance
\begin{align*}
	d([f_1],[f_2])
		= \inf_{x_0 \in \R^d}\|f_1(\,\cdot\, - x_0) - f_2\|_{H^1}
\end{align*}
(see \cite[Lemma A.1]{DuyHolRou08}).

\begin{proposition}[Precompact flow of the critical solution]\label{Precompact flow of the critical solution}
Let
$
	u_c
$
be the critical solution constructed in Theorem \ref{Existence of CS}.
Then, there exists a continuous path
$
	x(t)
$
in
$
	\R^d
$
such that
\begin{align*}
	K
		= \{u_c(t,\,\cdot\, - x(t)) : t \in [0,\infty)\}
		\subset H^1
\end{align*}
is precompact in
$
	H^1.
$
\end{proposition}

\begin{proof}
We deduce contradiction by assuming that Proposition \ref{Precompact flow of the critical solution} does not hold.
Then, \cite[Lemma A.3]{DuyHolRou08} deduces that
\begin{align*}
	\pi(\{u_c(t) : t \in [0,\infty)\})
\end{align*}
is not precompact in
$
	H^1/{\sim},
$
so there exists
$
	\{[u_c(\tau_n)]\}
$
such that any subsequence does not converge.
That is, for any
$
	[u^1]
		\in H^1/{\sim},
$
there exists
$
	\varepsilon
		> 0
$
such that for any
$
	N
		\in \N,
$
there exists
$
	n
		\geq N
$
such that
\begin{align}\label{140}
	\inf_{x_0 \in \R^d}\|u_c(\tau_n,\,\cdot\, - x_0) - u^1\|_{H^1}
		\geq \varepsilon.
\end{align}
In the case
$
	\{\tau_n\}
$
is a convergent sequence,
$
	u_c(\tau_n,\,\cdot\,)
$
converges since the solution to \eqref{NLS} is continuous with respect to
$
	t.
$
This contradicts \eqref{140}.
Thus, we consider a
$
	\{\tau_n\}
$
with
$
	\tau_n
		\longrightarrow \infty
$
as
$
	n
		\rightarrow \infty.
$
Applying the linear profile decomposition (Lemma \ref{LPD}) to
$
	\{u_c(\tau_n)\},
$
we have
\begin{align*}
	u_c(\tau_n,x)
		= \sum_{j=0}^Je^{-it_n^j\Delta}u^j(x-x_n^j) + U_n^J(x).
\end{align*}
The same argument as Theorem \ref{Existence of CS} deduces that
$
	J
		= 1
$
and hence, we write as
\begin{align*}
	u_c(\tau_n,x)
		= e^{-it_n \Delta}u^1(x-x_n) + U_n(x)
		= v^1(-t_n,x-x_n) + V_{n,0}(x),
\end{align*}
where
$
	v^1
$
and
$
	V_{n,0}
$
are the same those as in the proof of Theorem \ref{Existence of CS}.
Using the argument in Theorem \ref{Existence of CS} again, we exclude the case of
$
	t_n
		\longrightarrow - \infty
$
as
$
	n
		\rightarrow \infty.
$
Here, we remove also the case of
$
	t_n
		\longrightarrow \infty
$
as
$
	n
		\rightarrow \infty.
$
In this case, the inequality
\begin{align*}
	\|e^{it\Delta}u_c(\tau_n)\|_{S(-\infty,0)}
		\leq \|e^{it\Delta}u^1\|_{S(-\infty,-t_n)} + \|e^{it\Delta}U_n\|_{S(-\infty,0)}
		\longrightarrow 0
\end{align*}
holds.
Proposition \ref{SDS} shows that
$
	\|u_c\|_{S(-\infty,\tau_n)}
		= \|u_c(\cdot + \tau_n)\|_{S(-\infty,0)}
		\longrightarrow 0,
$
which gives us
$
	u_c
		\equiv 0.
$
This is contradiction.
Therefore, we obtain
$
	t_n
		= 0
$
for each
$
	n
		\in \N.
$

As the argument of (Step 3) in Theorem \ref{Existence of CS}, we have
\begin{align*}
	\sup_{t \in [0,\infty)}T_\omega(v^1)
		\leq \sup_{t \in [-\tau_n,\infty)}T_\omega(v^1)
		\leq \sup_{t \in [-\tau_n,\infty)}T_\omega(u_c(\cdot + \tau_n))
		= \sup_{t \in [0,\infty)}T_\omega(u_c)
		\leq B
\end{align*}
and
\begin{align*}
	S_\omega(u^1)
		= S_\omega(v^1)
		= S_\omega^c(B).
\end{align*}
In addition, it follows from \eqref{119} and \eqref{120} that
\begin{align*}
	\limsup_{n \rightarrow \infty}T_\omega(U_n)
		& \leq - T_\omega(u^1) + \limsup_{n \rightarrow \infty}T_\omega(u_c(\tau_n))
		< B
		< S_\omega(Q_\omega).
\end{align*}
Combining this with \eqref{136}, we have
$
	K(U_n)
		> 0
$
for sufficiently large
$
	n
		\in \N
$
and hence, we obtain
\begin{align*}
	\lim_{n \rightarrow \infty}\|u_c(\tau_n) - u^1(\cdot - x_n)\|_{H^1}^2
		& = \lim_{n \rightarrow \infty}\|U_n\|_{H^1}^2
		\lesssim \lim_{n \rightarrow \infty}S_\omega(U_n)
		= 0,
\end{align*}
where we refer to e.g., \cite[Proposition 4.2]{HamIkeISAAC} for the last inequality.
That is,
$
	u_c(\tau_n,\cdot\,+ x_n)
$
converges to
$
	u^1
$
in
$
	H^1(\R^d)
$
as
$
	n
		\rightarrow \infty,
$
which contradicts \eqref{140}.
\end{proof}

A similar argument to \cite[Proposition 4.1]{DuyHolRou08} deduces that the critical solution
$
	u_c
$
given in Theorem \ref{Existence of CS} has zero momentum, that is,
\begin{align*}
	\mathcal{M}(u_c(t))
		:= \Im\int_{\R^d}\overline{u_c(t,x)}\nabla u_c(t,x)dx
		\equiv 0.
\end{align*}

\begin{lemma}\label{UL and order of x}
Let
$
	u
$
be a forward time-global solution to \eqref{NLS} satisfying that
\begin{align*}
	K
		= \{u(t,\cdot - x(t)) : t \in [0,\infty)\}
		\subset H^1(\R^d)
\end{align*}
is precompact for some function
$
	x(t) : [0,\infty) \longrightarrow \R^d.
$
Then, for each
$
	\e
		> 0,
$
there exists
$
	R
		= R_{\e}
		> 0
$
such that
\begin{align*}
	& \int_{|x + x(t)| > R}(|u(t,x)|^2 + |\nabla u(t,x)|^2 + |u(t,x)|^{p+1})dx
		\leq \e
\end{align*}
for any
$
	t
		\in [0,\infty).
$
Moreover, if
$
	x(t)
$
is continuous and
$
	\mathcal{M}(u)
		= 0,
$
then
$
	\frac{x(t)}{t}
		\longrightarrow 0
$
as
$
	t
		\rightarrow \infty.
$
\end{lemma}

For the proof, see \cite[Corollary 3.3]{DuyHolRou08} and \cite[Lemma 5.1]{DuyHolRou08}.

\begin{lemma}\label{LB of K}
Let
$
	u
		\in C_t([0,\infty);H_x^1)
$
be a forward time-global solution to \eqref{NLS}.
Assume that there exists a path
$
	x(t)
$
in
$
	\R^d
$
such that
\begin{align*}
	K
		= \{u(t,\,\cdot\, - x(t)) : t \in [0,\infty)\}
			\subset H^1(\R^d)
\end{align*}
is precompact in
$
	H^1.
$
Then, there exists
$
	A
		> 0
$
such that
$
	A \cdot M(u)
		\leq L(u)
$
for any
$
	0
		\leq t
		< \infty.
$
\end{lemma}

For the proof, see e.g., \cite[Lemma 5.18]{HamArx}.

\begin{proposition}
There exists no the critical solution
$
	u_c
$
constructed in Theorem \ref{Existence of CS}.
\end{proposition}

\begin{proof}
By Lemma \ref{UL and order of x}, for any
$
	\eta
		> 0,
$
there exists
$
	T_0
		= T_0(\eta)
		> 0
$
such that
$
	|x(t)|
		\leq \eta t
$
for any
$
	t
		\geq T_0.
$
Let
$
	r
		= |x|.
$
A cut-off function
$
	\mathscr{X}_R
		\in C_0^\infty(\R^d)
$
is radially symmetric and satisfies
\begin{equation}
	\mathscr{X}_R(r)
		:= R^2\mathscr{X}\left(\frac{r}{R}\right),\ \text{ where }\ 
	\mathscr{X}(r) := \label{002}
		\left\{ \hspace{-0.2cm}
		\begin{array}{ll}
			r^2 & (0 \leq r \leq 1), \\
			\text{smooth} & (1 \leq r \leq 3), \\
			0 & (3 \leq r),
		\end{array}
		\right.
\end{equation}
$
	\mathscr{X}''(r)
		\leq 2
$
$
	(r \geq 0).
$
We set
\begin{align*}
	I(t)
		= \int_{\R^d}\mathscr{X}_R(x)|u_c(t,x)|^2dx.
\end{align*}
Then, it follows from that
\begin{align}
	|I'(t)|
		& \leq 2R\left|\int_{\R^d}\mathscr{X}'\left(\frac{r}{R}\right)\frac{x\cdot\nabla u_c(t,x)}{r}\overline{u_c(t,x)}dx\right| \notag \\
		& \leq c R\|u_c(t)\|_{H^1}^2
		\leq c R\{S_\omega^c(B) + B\}
		\leq cR \label{146}
\end{align}
for any
$
	0
		\leq t
		< \infty.
$
In addition, we have
\begin{align*}
	I''(t)
		& = 4K(u_c(t)) + \{R_1 - 8L(u_c(t))\} + R_2 + \left\{R_3 + \frac{4d(p-1)}{p+1}N(u_c(t))\right\} \\
		& =: 4K(u_c(t)) + \widetilde{R}_1 + R_2 + \widetilde{R}_3,
\end{align*}
where
$
	R_1, R_2, R_3
$
are defined as
\begin{align}\label{185}
\begin{split}
	R_1
		& = 4\int_{\R^d}\left\{\frac{1}{r^2}\mathscr{X}''\left(\frac{r}{R}\right) - \frac{R}{r^3}\mathscr{X}'\left(\frac{r}{R}\right)\right\}|x\cdot\nabla u_c(t,x)|^2 + \frac{R}{r}\mathscr{X}'\left(\frac{r}{R}\right)|\nabla u_c(t,x)|^2dx, \\
	R_2
		& = - \int_{\R^d}\left\{\frac{1}{R^2}\mathscr{X}^{(4)}\left(\frac{r}{R}\right) + \frac{2(d-1)}{Rr}\mathscr{X}^{(3)}\left(\frac{r}{R}\right) \right. \\
		& \qquad \left. + \frac{(d-1)(d-3)}{r^2}\mathscr{X}''\left(\frac{r}{R}\right) + \frac{(d-1)(3-d)R}{r^3}\mathscr{X}'\left(\frac{r}{R}\right)\right\}|u_c(t,x)|^2dx, \\
	R_3
		& = - \frac{2(p-1)}{p+1}\int_{\R^d}\left\{\mathscr{X}''\left(\frac{r}{R}\right) + \frac{(d-1)R}{r}\mathscr{X}'\left(\frac{r}{R}\right)\right\}|u_c(t,x)|^{p+1}dx.
\end{split}
\end{align}
Noting the inequality
\begin{align*}
	|\widetilde{R}_1 + R_2 + \widetilde{R}_3|
		\lesssim \int_{|x|\geq R}\(|\nabla u_c(t,x)|^2 + |u_c(t,x)|^{p+1} + \frac{1}{R^2}|u_c(t,x)|^2\)dx,
\end{align*}
we have
\begin{align*}
	I''(t)
		\geq 4K(u_c(t)) - C\int_{|x|\geq R}\(|\nabla u_c(t,x)|^2 + |u_c(t,x)|^{p+1} + \frac{1}{R^2}|u_c(t,x)|^2\)dx.
\end{align*}
Combining this inequality, Lemma \ref{Coercivity}, and Lemma \ref{LB of K}, it follows that
\begin{align*}
	I''(t)
		& \gtrsim M(u_c) - \int_{|x|\geq R}\(|\nabla u_c(t,x)|^2 + |u_c(t,x)|^{p+1} + \frac{1}{R^2}|u_c(t,x)|^2\)dx.
\end{align*}
By Lemma \ref{UL and order of x}, there exists
$
	R_0
		> 1
$
such that
\begin{align*}
	\int_{|x+x(t)| \geq R_0}\(|\nabla u_c(t,x)|^2 + |u_c(t,x)|^{p+1} + \frac{1}{R^2}|u_c(t,x)|^2\)dx
		< \frac{1}{2}M(u_c)
\end{align*}
for any
$
	0
		\leq t
		< \infty.
$
If we take
$
	R
		= R_0 + \sup_{t \in [T_0,T_1]}|x(t)|
		> 1,
$
then
\begin{align*}
	|x + x(t)|
		\geq |x| - |x(t)|
		\geq R - \sup_{[T_0,T_1]}|x(t)|
		= R_0
\end{align*}
for
$
	x
		\in \R^d
$
with
$
	|x|
		> R
$
and
$
	t
		\in [T_0,T_1],
$
where
$
	T_1
		> T_0
$
is chosen later.
For such
$
	R
		> 1
$
and
$
	t
		\in [T_0,T_1],
$
we have
\begin{align*}
	I''(t)
		\geq \frac{C}{2}M(u_c).
\end{align*}
Integrating this inequality in
$
	[T_0,T_1],
$
and using \eqref{146},
\begin{align*}
	\frac{C}{2}M(u_c)(T_1 - T_0)
		& \leq I'(T_1) - I'(T_0)
		\leq |I'(T_1)| + |I'(T_0)| \\
		& \leq 2cR
		= 2c\Bigl(R_0 + \sup_{t\in[T_0,T_1]}|x(t)|\Bigr)
		\leq 2c(R_0 + T_1\eta).
\end{align*}
This inequality is contradiction if we take
$
	\eta
		> 0
$
sufficiently small and
$
	T_1
		> 0
$
sufficiently large when
$
	M(u_c)
		> 0.
$
Therefore,
$
	M(u_c)
		= 0,
$
that is,
$
	u_c
		\equiv 0.
$
However, this contradicts that
$
	u_c
$
is a forward non-scattering solution.
\end{proof}

\section{Application to NLS with a potential}\label{Sec:Application}

In this section, we consider nonlinear Schr\"odinger equation with a potential
\begin{align}\label{NLSp}
	i\partial_t u + \Delta_P u
		= - |u|^{p-1}u,
	\qquad (t,x) \in \R \times \R^d,
\end{align}
where
$
	\Delta_P
		:= \Delta - P.
$
The equation \eqref{NLSp} has mass and energy defined as
\begin{align*}
	M(f)
		:= \|f\|_{L^2}^2
	\ \text{ and }\ 
	E_P(f)
		:= \frac{1}{2}L_{E_P}(f) - \frac{1}{p+1}N(f),
\end{align*}
where
\begin{align*}
	L_{E_P}(f)
		:= L(f) + \int_{\R^d}P(x)|f(x)|^2dx
\end{align*}
and
$
	L, N
$
are defined in \eqref{176}.
Then, we define action and virial functionals as follows:
\begin{align*}
	S_{\omega,P}(f)
		:= \frac{\omega}{2}M(f) + E_P(f)
	\ \text{ and }\ 
	K_P(f)
		:= 2L_{K_P}(f) - \frac{d(p-1)}{p+1}N(f),
\end{align*}
where
\begin{align*}
	L_{K_P}(f)
		:= L(f) - \frac{1}{2}\int_{\R^d}(x\cdot \nabla P)|f(x)|^2dx.
\end{align*}
Set
$
	P(x)
		:= \frac{\gamma}{|x|^\mu},
$
that is, we consider
\begin{align}\label{NLSv}
	i\partial_t u + \Delta u - \frac{\gamma}{|x|^\mu}u
		= - |u|^{p-1}u
\end{align}
and assume that
\begin{align*}
	\gamma
		> 0, \quad
	d
		= p
		= 3, \quad
	0
		< \mu
		< 2
\end{align*}
throughout this section.
We note that if
$
	P(x)
		= \frac{\gamma}{|x|^\mu},
$
then
\begin{align*}
	- \int_{\R^d}(x\cdot \nabla P)|f(x)|^2dx
		= \mu\int_{\R^d}\frac{\gamma}{|x|^\mu}|f(x)|^2dx.
\end{align*}
For these settings, we consider minimizing problems:
\begin{align*}
	r_{\omega,P}
		& := \inf\{S_{\omega,P}(f) : f \in H_{\rad}^1(\R^d) \setminus \{0\},\ K_P(f) = 0\}, \\
	n_{\omega,P}
		& := \inf\{S_{\omega,P}(f) : f \in H^1(\R^d) \setminus \{0\},\ K_P(f) = 0\}.
\end{align*}
Then, the following results are known in \cite[Theorem 1.5]{HamIkeISAAC}.

\begin{proposition}
$
	r_{\omega,P}
$
is attained.
Moreover,
$
	n_\omega
		= n_{\omega,P}
		< r_{\omega,P}
$
holds, where
$
	n_\omega
$
is defined in \eqref{nw}.
\end{proposition}

From now on, we write
$
	Q_{\omega,P}
$
to denote an optimizer to
$
	r_{\omega,P}
$
throughout the section.
Based on these settings, we show Theorem \ref{S vs B P}.
The flow of proof is similar to that of Theorem \ref{S vs B}.

At first, we only note the following for the global existence and blow-up results: We use
\begin{align*}
	\int_{\R^d}\left\{\frac{1}{2}P + \frac{1}{d(p-1)}x\cdot \nabla P\right\}|u(t,x)|^2dx
		= \left\{\frac{1}{2} - \frac{\mu}{d(p-1)}\right\}\int_{\R^d}\frac{\gamma}{|x|^\mu}|u(t,x)|^2dx
		\geq 0
\end{align*}
to deduce uniformly boundedness of
$
	L(u(t))
$
(which corresponds to \eqref{181}).
In addition, we make use of
\begin{align*}
	\int_{\R^d}(2P + x\cdot \nabla P)|u(t,x)|^2dx
		= (2 - \mu)\int_{\R^d}\frac{\gamma}{|x|^\mu}|u(t,x)|^2dx
		\geq 0
\end{align*}
to see that the global solution satisfies
\begin{align*}
	\sup_{t \in [t_0,\infty)}T_\omega(u(t))
		< S_{\omega,P}(Q_{\omega,P})
\end{align*}
for any
$
	t_0
		\in (0,\infty)
$
(which corresponds to Remark \ref{Uniformly bounded}),
where
$
	T_\omega
$
is defined in \eqref{142}.

We define spaces
\begin{align*}
	X(I)
		& := L_t^\infty(I;L_x^2) \cap L_t^2(I;L_x^6), \quad
	S(I)
		:= L_t^5(I;L_x^5), \\
	N(I)
		& := {L_t^\frac{5}{3}(I;L_x^\frac{30}{23})}
\end{align*}
for a time interval
$
	I.
$
In addition, we put a space
$
	X^1(I)
$
with a norm
$
	\|f\|_{X^1(I)}
		:= \|(1-\Delta)^\frac{1}{2}f\|_{X(I)}.
$
Then, we can find the following in \cite[Proposition 4.7]{HamIkeJMAA}.

\begin{proposition}\label{S-norm P}
Assume that the solution
$
	u
$
to \eqref{NLSv} exists on
$
	[0,\infty)
$
and has
\begin{align*}
	\|u\|_{S([0,\infty))}
		< \infty,
	\quad
	\|u\|_{L_t^\infty([0,\infty);H_x^1)}
		< \infty.
\end{align*}
Then,
$
	u
$
satisfies
\begin{align*}
	\|u\|_{X^1([0,\infty))}
		< \infty
\end{align*}
and scatters in the forward time.
\end{proposition}

From now on, we see the scattering result in Theorem \ref{S vs B P}.
It suffices to prove the next scattering criterion.

\begin{theorem}[Scattering criterion]
Let
$
	u
		\in C_t([0,\infty);H_{\rad}^1(\R^3))
$
be a forward time global solution to \eqref{NLSv}.
If there exists
$
	\omega
		> 0
$
such that
\begin{align*}
	\sup_{t\geq 0}T_\omega(u(t))
		< S_{\omega,P}(Q_{\omega,P}),
\end{align*}
then
$
	u
$
scatters in the forward time.
\end{theorem}

We define
\begin{align*}
	S_{\omega,P}^c(A,B)
		:=
		\sup\left\{
		\|u\|_{S([0,\infty))}
		\left|
		\begin{array}{l}
		\text{The solution }
		u
			\in C_t([0,\infty);H_{\rad}^1(\R^3))
		\text{ satisfies } \\
		S_{\omega,P}(u_0)
			\leq A
		\text{ and }
		\sup_{t \in [0,\infty)}T_\omega(u(t))
			\leq B.
		\end{array}
		\right.
		\right\}
\end{align*}
for
$
	A, B
		> 0.
$
Then, we would like to show the following:
\begin{theorem}\label{SC P}
Suppose that
$
	A
		> 0
$
and
$
	0
		< B
		< S_{\omega,P}(Q_{\omega,P}).
$
Then, we have
$
	S_{\omega,P}^c(A,B)
		< \infty.
$
\end{theorem}

To prove Theorem \ref{SC P}, we assume that
$
	S_{\omega,P}(A_\ast,B_\ast)
		= \infty
$
for some
$
	(A_\ast,B_\ast)
		\in (0,\infty) \times (0,S_{\omega,P}(Q_{\omega,P}))
$
and deduce contradiction.
Then,
\begin{align*}
	S_{\omega,P}^c
		= S_{\omega,P}^c(B)
		:= \inf\{A > 0 : S_{\omega,P}(A,B) = \infty\}
\end{align*}
is well-defined for
$
	[B_\ast,S_{\omega,P}(Q_{\omega,P})).
$

\begin{lemma}\label{Coercivity P}
Let
$
	0
		< B
		< S_{\omega,P}(Q_{\omega,P}).
$
There exists
$
	C
		> 0
$
such that, for any
$
	f
		\in H_{\rad}^1(\R^3)
$
with
$
	T_\omega(f)
		\leq B,
$
we have
\begin{align*}
	K_P(f)
		\geq CL_{E_P}(f)
	\ \text{ and }\ 
	E_P(f)
		\geq \frac{C}{4}L_{E_P}(f).
\end{align*}
\end{lemma}

\begin{proof}
Since translation and the potential
$
	\frac{\gamma}{|x|^\mu}
$
are incompatible, it possible that we cannot make use of Lieb's compactness lemma.
Then, we need to modify the proof of Lemma \ref{Coercivity} slightly and rely on Strauss's radial compactness.

The first half claim that
$
	``L_{E_P}(f)
		\leq \frac{\mu}{2}C_{\omega,B}
$
deduces
$
	K_P(f)
		\geq \frac{\mu}{2}L_{E_P}(f)
$
for
$
	C_{\omega,B}^{-\frac{1}{2}}
		:= \frac{3}{2}C_{\rm GN}\(\frac{2B}{\omega}\)^\frac{1}{2}"
$
is similar to Lemma \ref{Coercivity}

On the other hand, if we deny the desired result when
$
	L_{E_P}(f)
		> \frac{\mu}{2}C_{\omega,B},
$
there exists
$
	\{f_n\}
		\subset H_{\rad}^1(\R^3)
$
such that
\begin{align}
	& L_{E_P}(f_n)
		> \frac{\mu}{2}C_{\omega,B}, \notag \\
	& T_\omega(f_n)
		\leq B
		< S_{\omega,P}(Q_{\omega,P}), \label{170} \\
	& K_P(f_n)
		< \frac{\mu}{2n}L_{E_P}(f_n). \label{171}
\end{align}
\eqref{171} implies that
\begin{align}\label{172}
	\mu\(1 - \frac{1}{2n}\)L_{E_P}(f_n)
		< \frac{3}{2}N(f_n).
\end{align}
Combining \eqref{170} and \eqref{172}, we notice that
$
	\sup_{n \in \N}\|f_n\|_{H^1}
		< \infty
$
holds and hence, we have
$
	K_P(f_n)
		\longrightarrow 0
$
as
$
	n
		\rightarrow \infty.
$
By Strauss's radial compactness, we assume that
\begin{align*}
	& f_n
		\rightharpoonup f_\infty \text{ in }H^1(\R^3), \\
	& f_n
		\longrightarrow f_\infty \text{ in }L^4(\R^3), \\
	& f_n
		\longrightarrow f_\infty \text{ a.e. }x \in \R^3
\end{align*}
as
$
	n
		\rightarrow \infty
$
by passing to a subsequence.
The rest of the proof is similar to that of Lemma \ref{Coercivity}.
\end{proof}

\begin{theorem}\label{Existence of CS P}
Assume that
$
	B
		\in (0,S_{\omega,P}(Q_{\omega,P}))
$
satisfy
$
	S_{\omega,P}^c(B)
		< \infty.
$
Then, there exits a solution
$
	u_c
$
to \eqref{NLSv} such that
$
	S_{\omega,P}(u_c)
		= S_{\omega,P}^c(B),
$
$
	\sup_{t \in [0,\infty)}T_\omega(u_c(t))
		\leq B,
$
and
$
	u_c
$
does not scatter in the forward time.
\end{theorem}

The proof is based on an argument combined that of Theorem \ref{Existence of CS} and \cite[Lemma 6.1]{HamIkeJMAA}.

\begin{proposition}[Precompact flow of the critical solution]\label{Precompact flow of the critical solution P}
Let
$
	u_c
$
be a critical solution constructed in Proposition \ref{Existence of CS P}.
Then,
\begin{align*}
	K
		= \{u_c(t,\,\cdot\,) : t \in [0,\infty)\}
		\subset H_{\rad}^1
\end{align*}
is precompact in
$
	H^1.
$
\end{proposition}

\begin{proof}
Unlike Proposition \ref{Precompact flow of the critical solution}, the proposition is shown directly, not the contradiction.
However, it is similar to the argument after \eqref{140}.
\end{proof}

The next two lemmas are shown as Lemmas \ref{UL and order of x} and \ref{LB of K} respectively.

\begin{lemma}\label{UL P}
Let
$
	u
$
be a forward time-global solution to \eqref{NLSv} satisfying that
\begin{align*}
	K
		= \{u(t,\cdot) : t \in [0,\infty)\}
		\subset H_{\rad}^1(\R^d)
\end{align*}
is precompact.
Then, for each
$
	\e
		> 0,
$
there exists
$
	R
		= R_{\e}
		> 0
$
such that
\begin{align*}
	& \int_{|x| > R}\(|u(t,x)|^2 + |\nabla u(t,x)|^2 + \frac{\gamma}{|x|^\mu}|u(t,x)|^2 + |u(t,x)|^4\)dx
		\leq \e
\end{align*}
for any
$
	t
		\in [0,\infty).
$
\end{lemma}

\begin{lemma}\label{LB of K P}
Let
$
	u
		\in C_t([0,\infty);H_x^1)
$
be a forward time-global solution to \eqref{NLSv}.
Assume that
\begin{align*}
	K
		= \{u(t,\,\cdot\,) : t \in [0,\infty)\}
			\subset H^1(\R^3)
\end{align*}
is precompact in
$
	H^1.
$
Then, there exists
$
	A
		> 0
$
such that
$
	A \cdot M(u)
		\leq L_{E_P}(u)
$
for any
$
	0
		\leq t
		< \infty.
$
\end{lemma}

\begin{proposition}
There exists no the critical solution
$
	u_c
$
constructed in Theorem \ref{Existence of CS P}.
\end{proposition}

\begin{proof}
Let
$
	r
		= |x|.
$
A cut-off function
$
	\mathscr{X}_R
		\in C_0^\infty(\R^3)
$
is as in \eqref{002}.
We set
\begin{align*}
	I(t)
		= \int_{\R^3}\mathscr{X}_R(x)|u_c(t,x)|^2dx.
\end{align*}
Then, it follows from that
\begin{align}
	|I'(t)|
		& \leq 2R\left|\int_{\R^3}\mathscr{X}'\left(\frac{r}{R}\right)\frac{x\cdot\nabla u_c(t,x)}{r}\overline{u_c(t,x)}dx\right| \notag \\
		& \leq c R\|u_c(t)\|_{H^1}^2
		\leq c R\{S_{\omega,P}^c(B) + B\}
		\leq cR \label{166}
\end{align}
for any
$
	0
		\leq t
		< \infty.
$
In addition, we have
\begin{align*}
	I''(t)
		& = 4K_P(u_c(t)) + \{R_1 - 8L_{E_P}(u_c(t))\} + R_2 \\
		& \qquad + \left\{R_3 + 6N(u_c(t))\right\} + \left\{R_4 - 4\mu\int_{\R^3}\frac{\gamma}{|x|^\mu}|u_c(t,x)|^2dx\right\} \\
		& =: 4K_P(u_c(t)) + \~{R}_1 + R_2 + \~{R}_3 + \~{R}_4,
\end{align*}
where
$
	R_j
$
$
	(1\leq j\leq 3)
$
are defined in \eqref{185} with
$
	d
		= p
		= 3
$
and
$
	R_4
$
is
\begin{align*}
	& R_4
		= 2\mu\int_{\R^3}\frac{R}{r}\mathscr{X}'\(\frac{r}{R}\)\frac{\gamma}{|x|^\mu}|u_c(t,x)|^2dx.
\end{align*}
Noting the inequality
\begin{align*}
	|\~{R}_1 + R_2 + \~{R}_3 + \~{R}_4|
		\lesssim \int_{|x|\geq R}\(|\nabla u_c(t,x)|^2 + |u_c(t,x)|^4 + \frac{1}{R^\mu}|u_c(t,x)|^2\)dx,
\end{align*}
we have
\begin{align*}
	I''(t)
		\geq 4K_P(u_c(t)) - C\int_{|x|\geq R}\(|\nabla u_c(t,x)|^2 + |u_c(t,x)|^4 + \frac{1}{R^\mu}|u_c(t,x)|^2\)dx.
\end{align*}
Combining this inequality, Lemma \ref{Coercivity P}, and Lemma \ref{LB of K P}, it follows that
\begin{align*}
	I''(t)
		& \gtrsim M(u_c) - \int_{|x|\geq R}\(|\nabla u_c(t,x)|^2 + |u_c(t,x)|^{p+1} + \frac{1}{R^\mu}|u_c(t,x)|^2\)dx.
\end{align*}
By Lemma \ref{UL P}, there exists
$
	R_0
		> 1
$
such that
\begin{align*}
	\int_{|x|\geq R_0}\(|\nabla u_c(t,x)|^2 + |u_c(t,x)|^{p+1} + \frac{1}{R^\mu}|u_c(t,x)|^2\)dx
		< \frac{1}{2}M(u_c)
\end{align*}
for any
$
	0
		\leq t
		< \infty.
$
Taking a such
$
	R_0
		> 1,
$
we have
\begin{align*}
	I''(t)
		\geq \frac{C}{2}M(u_c)
		> 0
\end{align*}
for any
$
	t
		\in [0,\infty).
$
This implies that
$
	I'(t)
		\longrightarrow \infty
$
as
$
	t
		\rightarrow \infty.
$
However, this contradicts \eqref{166}.
\end{proof}

\appendix

\section{Alternative proof of a coercivity}\label{Alternative}

In this section, we provide an alternative proof of Lemma \ref{Coercivity}.
Unlike its proof in Section \ref{Sec: S}, we use the scaling property of the equation and find an explicit constant of the inequality.
In particular, we prove the following:
\begin{lemma}
Let
$
	0
		< B
		< S_\omega(Q_\omega).
$
For any
$
	f
		\in H^1(\R^d)
$
with
$
	T_\omega(f)
		\leq B,
$
we have
\begin{align*}
	K(f)
		\geq 2\left\{1 - \(\frac{B}{S_{\omega}(Q_{\omega})}\)^\frac{2}{d}\right\}L(f)
	\ \text{ and }\ 
	E(f)
		\geq \frac{1}{2}\left\{1 - \(\frac{B}{S_{\omega}(Q_{\omega})}\)^\frac{2}{d}\right\}L(f).
\end{align*}
\end{lemma}

\begin{proof}
As in Lemma \ref{Coercivity}, we prove only the first inequality.
We note that it follows from the Gagliardo-Nirenberg inequality \eqref{G-N inequality} that
\begin{align*}
	N(f)^\frac{2}{(p-1)s_c+2}
		\leq C_{\rm GN}^\frac{2}{(p-1)s_c+2}M(f)^\frac{(p-1)(1-s_c)}{(p-1)s_c+2}L(f).
\end{align*}
Thus, we have
\begin{align*}
	K(f)
		& \geq \left\{2 - \frac{d(p-1)}{p+1}C_{\rm GN}^\frac{2}{(p-1)s_c+2}M(f)^\frac{(p-1)(1-s_c)}{(p-1)s_c+2}N(f)^\frac{(p-1)s_c}{(p-1)s_c+2}\right\}L(f).
\end{align*}
We estimate the second term by using
\begin{align*}
	\left\{\frac{\omega}{2(1-s_c)}M(f)\right\}^{1-s_c}\left\{\frac{p-1}{2(p+1)}N(f)\right\}^{s_c}
		\leq T_\omega(f)
		\leq B.
\end{align*}
Then, we have
\begin{align*}
	K(f)
		& \geq \[2 - \frac{d(p-1)}{p+1}C_{\rm GN}^\frac{2}{(p-1)s_c+2}\left\{\frac{2(1-s_c)}{\omega}\right\}^\frac{(p-1)(1-s_c)}{(p-1)s_c+2}\left\{\frac{2(p+1)}{p-1}\right\}^\frac{(p-1)s_c}{(p-1)s_c+2}B^\frac{p-1}{(p-1)s_c+2}\]L(f).
\end{align*}
The coefficient of
$
	B
$
is calculated by the Pohozaev identity \eqref{PI} as
\begin{align*}
	\left\{\frac{d(p-1)}{p+1}\right\}^\frac{(p-1)s_c+2}{p-1}C_{\rm GN}^\frac{2}{p-1}\left\{\frac{2(1-s_c)}{\omega}\right\}^{1-s_c}\left\{\frac{2(p+1)}{p-1}\right\}^{s_c}
		= 2^\frac{(p-1)s_c+2}{p-1}S_\omega(Q_\omega)^{-1},
\end{align*}
which implies the desired result.
\end{proof}

\subsection*{Acknowledgements}
The author is supported by JSPS KAKENHI Grant Number 25K17285 and JST CREST Grant Number JPMJCR24Q6.

\end{document}